\documentclass[11pt, reqno]{amsart}
\usepackage[margin=1in]{geometry}
\usepackage{amsmath,amssymb}
\usepackage{xcolor} 
\usepackage{amsthm}
\usepackage{amsfonts}
\usepackage{subcaption}
\usepackage{graphicx}
\usepackage[dvpis]{epsfig}
 \usepackage[colorlinks=true]{hyperref}
 \hypersetup{urlcolor=blue, citecolor=red}
\makeatletter
\newcommand*{\rom}[1]{\expandafter\@slowromancap\romannumeral #1@}
\makeatother

\newcommand{\R}{\mathbb{R}}

\newtheorem{theorem}{Theorem}[section]

\newtheorem{lemma}[theorem]{Lemma}

\newtheorem{proposition}[theorem]{Proposition}

\theoremstyle{definition}

\theoremstyle{remark}

\newtheorem{remark}[theorem]{Remark}

\numberwithin{equation}{section}

\newcommand{\pa}{\partial}

\newcommand{\lt}{\left}
\newcommand{\rt}{\right}
\newcommand{\bq}{\begin{equation}}
\newcommand{\eq}{\end{equation}}

\newcommand{\calC}{\mathcal C}

\newcommand{\calL}{\mathcal L}

\newcommand{\calO}{\mathcal O}

\newcommand{\intr}{\int_\R}
 
\renewcommand{\geq}{\geqslant}
\renewcommand{\ge}{\geqslant}
\renewcommand{\leq}{\leqslant}
\renewcommand{\le}{\leqslant}
\newcommand{\dx}{\textnormal{d}x}
\newcommand{\dt}{\textnormal{d}t}
\newcommand{\dta}{\textnormal{d}\tau}
\newcommand{\ds}{\textnormal{d}s}
\newcommand{\cmax}{c_+}
\newcommand{\cmin}{c_-}
\newcommand{\gmax}{\gamma_+}
\newcommand{\gmin}{\gamma_-}

\usepackage{soul}
\usepackage[normalem]{ulem}
\usepackage{cancel}

\renewcommand{\sout}[1]{}
\renewcommand{\cancel}[1]{}

\makeatletter
\@namedef{subjclassname@2020}{\textup{2020} Mathematics Subject Classification}
\makeatother

\begin{document}

\title[Critical thresholds in pressureless Euler--Poisson equations]{Critical thresholds in pressureless Euler--Poisson equations with background states}

\author{Young-Pil Choi}
\address{Department of Mathematics, Yonsei University, Seoul 03722, Republic of Korea}
\email{ypchoi@yonsei.ac.kr}

\author{Dong-ha Kim}
\address{Department of Mathematics, Yonsei University, Seoul 03722, Republic of Korea}
\email{skyblue898@yonsei.ac.kr}

\author{Dowan Koo}
\address{Department of Mathematics, Yonsei University, Seoul 03722, Republic of Korea}
\email{dowan.koo@yonsei.ac.kr}

\author{Eitan Tadmor}
\address{Department of Mathematics and Institute for Physical Science and Technology, University of Maryland, College Park, MD 20742, USA}
\email{tadmor@umd.edu}

\date{\today}


\thanks{\textbf{Acknowledgment.} 
The authors thank Junsik Bae, Bongsuk Kwon, and Changhui Tan for helpful discussion on the Euler--Poisson system for cold ion dynamics and Euler--Poisson--alignment system. The research of YPC was supported by  NRF grant 2022R1A2C1002820. The research of ET was supported in part by ONR grant N00014-2112773.}

\begin{abstract}
We investigate the critical threshold phenomena in a large class of one dimensional pressureless Euler--Poisson (EP) equations, with non-vanishing background states. First, we establish  local-in-time well-posedness 
in proper regularity spaces, which are adapted for a certain \emph{neutrality condition} to hold. The neutrality condition is shown to be necessary: we construct smooth solutions that exhibit instantaneous failure of the neutrality condition, which in turn yields non-existence of solutions, even locally in time, in the classical Sobolev spaces $H^s(\R)$, $s \geq 2$. 
Next, we study the critical threshold phenomena in the neutrality-condition-satisfying pressureless EP systems, where we distinguish between two cases. We prove that in the case of attractive forcing, the neutrality condition can further restrict the sub-critical region into its borderline, namely --- the sub-critical region is reduced to a single line in the phase plane. We then turn to provide a rather definitive answer for the critical thresholds in  the case of repulsive EP systems with variable backgrounds. As an application, we analyze the critical thresholds for the damped EP system for cold plasma ion dynamics, where the density of electrons is given by the \emph{Maxwell--Boltzmann relation}.
\end{abstract}


\maketitle

\setcounter{tocdepth}{1}
\tableofcontents
%
%
%
%
%

\section{Introduction and statement of main results}
\subsection{Systems and notions}
In this paper, we are interested in the analysis of critical threshold phenomena in a large class of pressureless Euler--Poisson (EP) equations in one dimension. More precisely, the main purpose of our work is to propose a new method based on Lyapunov functions to investigate the super-critical region with finite-time breakdown and the sub-critical region with global regularity of $\calC^1$ solutions to the following equations:
\begin{equation}\label{main_sys}
\begin{cases}
\partial_t \rho + \partial_x(\rho u) = 0, \\
\partial_t u + u \partial_x u = -\nu u - k\partial_x \phi,\\
- \partial_{xx}\phi = \rho - c
\end{cases}
\end{equation}
in $\R\times (0,T)$ subject to initial data
\begin{equation}\label{ini}
(\rho, u)(0,x)=(\rho_0, u_0)(x), \quad x \in \R.
\end{equation}
Here, $\rho=\rho(t,x)$, $u=u(t,x)$ and $\phi = \phi(t,x)$ are respectively, the density, velocity,  and  potential of the flow.  The parameter $\nu\ge0$ indicates the strength of damping. The parameter $k$ on the right of  \eqref{main_sys}${}_2$ is a physical constant, signifying the underlying forcing of the system ---  either an \emph{attractive} or \emph{repulsive} forcing depending on whether  $k <0$ or, respectively, $k>0$; for simplicity, we use the re-scaled parameter $k=\pm1$ throughout the paper.\newline
 Equation \eqref{main_sys}${}_3$ involves a positive background state, $c = c(t,x)$, which is assumed to be uniformly bounded away from vacuum, 
\[
 0 <  c_{-} \leq c(t,x) \leq c_{+}, \quad \forall  (t,x) \in \R_+ \times \R,
\]
to vary smoothly in a manner specified in \eqref{cspace} below, and to preserve its  \emph{mass fluctuation}  in the sense that
 \begin{equation}\label{comp}
 \int_{\R} \lt(c(t,x) -c_0(x)\rt) \dx = 0,\quad \forall t\in [0,T).
\end{equation}
Finally, we augment \eqref{main_sys} with the   following \emph{neutrality condition}
\bq\label{def_neu}
\rho(t,\cdot) - c(t,\cdot) \in L^1(\R) \quad \mbox{and} \quad \intr (\rho(t,x) - c(t,x))\,\dx = 0 \qquad \forall t\in [0,T).
\eq
This neutrality condition plays a central role in our discussion. Accordingly,  we define a \emph{classical solution} to be a pair, $(\rho,u)$, satisfying both --- the system \eqref{main_sys} pointwise, and the neutrality condition \eqref{def_neu}.
It is essential to observe that the neutrality condition is independent of the main system \eqref{main_sys}. Indeed,  in Theorem \ref{anomal} below we construct so-called \textit{anomalous solutions}, in the sense of solving the main system \eqref{main_sys}  with neutrality-satisfying initial data, yet these solutions fail \eqref{def_neu}  instantaneously after the initial time.
Observe that since $c(t,x)\geq c_->0$, then total mass $\displaystyle \int_{\R} \rho(t,x)\dx$ is inifinte;  we impose  \eqref{comp}  as a \emph{compatiblity condition} with the neutrality  \eqref{def_neu}, so that the  total \emph{mass fluctuation},  $\displaystyle \int_{\R} \lt(\rho(t,x)-\rho_0(x)\rt)\dx$, is conserved in time.
\newline

\subsection{History and review}

The system \eqref{main_sys} is one of the fundamental fluid models describing the dynamics of many important physical flows, including charge transport, plasma waves, and cold ions ($k > 0$) as well as collapse of stars due to self-gravitation ($k < 0$), see \cite{BK19, BK22, BW98, DLYY02, ELT01, HJL81, MRS90} and references therein.

A mathematical study of finite-time loss of regularity (or singularity formation as shocks) and global-in-time regularity of solutions to Euler equations with nonlocal interaction forces are by now well established. Since the literature-related results are too vast to be mentioned here, we focus on results for the EP systems. The critical thresholds in the 1D pressureless EP system \eqref{main_sys} with, $\nu\geq0$, and $c = \bar c \geq 0$ are studied in \cite{ELT01}. The damped EP system or EP system with the quadratic confinement is also taken into account in \cite{BL20, CCZ16}. In the multi-dimensional case, the critical thresholds are investigated for the restricted EP system in \cite{LT02, LT03} and EP system with radial symmetry in \cite{BLpre, T21, WTB12}.

The case with pressure i.e., the pressure term $p(\rho) = \rho^\gamma$ with $\gamma \geq 1$ is added to the EP system, but with $c = 0$ and $\nu = 0$, is also studied in \cite{TW08}. The critical thresholds for the EP system with pressure and nonzero background state are a challenging open problem. We refer to \cite{GHZ17} for the global-in-time existence of solutions with small amplitude to the 1D EP system for electrons. The critical thresholds for the 1D Euler-alignment or Euler--Poisson-alignment systems are analyzed in \cite{BLTpre, CCTT16, TT14}. The finite-time singularity formation for the 1D EP system for cold ions is also discussed in \cite{BCKpre, L06}.


Beyond those developments, there is little literature on critical thresholds in EP systems with 
\emph{variable} background. In this context, we mention the recent work \cite{BL20_2} in which critical thresholds conditions in the system \eqref{main_sys} with the attractive forces are investigated. In order to handle new difficulties arising from the consideration of the variable background, a combination of phase plane analysis and comparison principle is used in \cite{BL20_2}. However, as mentioned in \cite{BL20_2}, that strategy would not be worked in the presence of repulsive forces.

%
%
%
%
%

\noindent
\subsection{Main results}
In this section, we present our main results on Euler--Poisson equation with background states. The purpose of this work is four-fold:\newline
(i) we present a new framework for the well-posedness theory for the EP system with the background, which includes a rigorous treatment of the neutrality condition \eqref{def_neu}, while justifying this framework by presenting concrete examples; \newline 
(ii) we point out that in the attractive case, the set of relevant configurations that yield global smooth solutions --- with variable and even with constant background, is, in fact, restricted to the borderline case due to the neutrality condition;\newline
(iii) we present an in-depth threshold analysis for global solutions of the repulsive EP system with variable background; and\newline
(iv) as an application of repulsive case, we study the case of cold plasma corresponding to exponential background $c=e^\phi$. 

Below, we present the main results in more detail.

\subsubsection{\textbf{Local well-posedness and neutrality condition}}\label{ssec:lwp}

We first develop a local-in-time well-posedness theory for the EP system with a non-vanishing background state \eqref{main_sys}.
The EP system with zero-background was studied by many authors, and in particular we refer to the $H^s(\R)$ well-posedness in \cite[Appendix]{LZ02}. However,  we cannot specify an appropriate reference for the local-in-time well-posedness theory for the EP system with \emph{non-vanishing} background, \eqref{main_sys}. Here, the neutrality condition, \eqref{def_neu}, implies the non-integrability $\rho(t,\cdot) \notin L^1(\R)$, in order to match the infinite mass that `enters' the system from infinity, $\displaystyle  \int_\R c(t,\cdot)\dx=\infty$. 
Consequently, the local well-posedness of \eqref{main_sys} requires a careful study of proper regularity space, beyond $H^s(\R)$. This is the content of our first theorem. We fix $s \geq 2$ throughout the paper.

\begin{theorem}[{\bf Local well-posedness}]\label{thm_apriori} Suppose that the initial data $(\rho_0, u_0)$ satisfy
\bq\label{eq:ICa}
(\rho_0 - c_0, u_0) \in H^s\cap \dot H^{-1}(\R) \times H^{s+1}(\R).
\eq
Then there exists a positive constant $T>0$ such that the system \eqref{main_sys}--\eqref{ini} admits a unique solution $(\rho, u)$ satisfying
\bq\label{cls:lwp}
(\rho - c, u) \in L^\infty(0,T;H^s\cap \dot H^{-1}(\R)) \times L^\infty(0,T; H^{s+1}(\R)).
\eq
Moreover, if the solutions $\rho$ and $\pa_x u$ are bounded over the time interval $[0,T]$, i.e. 
\[
|\|\{\rho,u\}|\|_{[0,T]} := \sup_{0 \leq t \leq T}\lt(\|\rho(t)\|_{L^\infty} + \|\pa_x u(t)\|_{L^\infty}\rt) < \infty,
\]
then there is propagation of higher order regularity of $(\rho, u)$ expressed in terms of\newline  $|\|\{\rho(t,\cdot),u(t,\cdot)\}|\|_s:=
\|(\rho - c)(t)\|_{H^s\cap \dot H^{-1}} + \|u(t)\|_{H^{s+1}}$,
 
\bq\label{est_X}
\sup_{0 \leq t \leq T}|\|\{\rho(t,\cdot),u(t,\cdot)\}|\|_s \leq C|\|\{\rho_0(\cdot),u_0(\cdot)\}|\|_s e^{\eta T},
\eq
where $\eta \sim |\|\{\rho,u\}|\|_{[0,T]} + \|\pa_t c\|_{L^\infty(0,T;H^s\cap \dot H^{-1})} + \|\pa_x c\|_{L^\infty(0,T;H^s)} + 1$.
\end{theorem}

We emphasize that the local well-posedness stated in Theorem \ref{thm_apriori} does not necessarily require the neutrality condition \eqref{def_neu} to hold. On the other hand, the following result shows that it is impossible to construct the solution map without $\dot{H}^{-1}$ in \eqref{cls:lwp}. Thus, Theorems \ref{thm_apriori} and \ref{illposed} tell us that it is essential to equip $\rho_0-c_0$ with additional $\dot H^{-1}$ regularity to secure the local well-posedness.

\begin{theorem} [{\bf Non-existence}] \label{illposed}
For the EP system \eqref{main_sys}--\eqref{ini} with $c = \bar{c} > 0$, there exists $(\rho_0,u_0)$ satisfying
\begin{equation*}
        (\rho_0 -\bar{c},u_0) \in H^s(\R)\times H^{s+1}(\R)
\end{equation*}
such that there is no solution to \eqref{main_sys}--\eqref{ini} with initial data $(\rho_0,u_0)$, which fulfills
\[
(\rho -\bar c , u ) \in L^{\infty}\lt([0,\delta];  H^s(\R) \times H^{s+1}(\R)\rt)
\]
for any $\delta >0$.
\end{theorem}
\noindent
In particular, Theorem \ref{illposed} implies that the EP system \eqref{main_sys}--\eqref{ini} is \emph{ill-posed} in $H^s\times H^{s+1}(\R)$. 

\medskip
Our next theorem identifies a proper subset of initial configurations \eqref{eq:ICa}, which secure the propagation of neutrality, a question that seems to be mostly ignored in the literature.
To provide a rigorous treatment of the extended system \eqref{main_sys},\eqref{def_neu}, we introduce proper regularity conditions for the background state $c(t,x)$,  given by
\begin{align}\label{cspace}
\left\{\begin{aligned}
      &\partial_t c \in L^{\infty}(\R_+ ; H^s\cap \dot H^{-1}\cap L^1(\R)), \\
     & \partial_x c \in L^{\infty}(\R_+ ; H^s(\R)).
\end{aligned}\right.
\end{align} 
In particular, the regularity imposed in \eqref{cspace}${}_1$ implies that the background function $c$ satisfies the compatibility condition \eqref{comp}. 
Indeed,  any function $f$ in $\dot{H}^{-1}\cap L^1(\R)$,  has a continuous Fourier transform $\widehat{f}(\cdot)$ such that $|\xi|^{-1}\widehat{f}(\xi)\in L^2(\R)$, \emph{and} in particular, it delivers that $\widehat{f}(0)=0$, i.e., $\displaystyle \intr f(x) \,\dx = 0$. Hence, we deduce $\displaystyle     \int_{\R}\partial_t c(t,x)\dx=0 \ \mbox{a.e.} \ t\in \R_+$, and by Fubini's theorem, we obtain the compatibility \eqref{comp},
\begin{align*}
\begin{aligned}
    \int_{\R}(c(t,x)-c_0(x))\dx&=\int_{\R}\int_{0}^{t}\partial_t c(t,x)\dt\dx =\int_{0}^{t}\int_{\R}\partial_t c(t,x)\dx\dt =0,\quad\forall t>0.
\end{aligned}
\end{align*}


\begin{theorem}[{\bf Well-posedness with neutrality}]\label{thm_neu_apri} 
Suppose that the initial data $(\rho_0,u_0)$ satisfy
\bq\label{eq:ICb}
(\rho_0 - c_0, u_0) \in H^s\cap \dot H^{-1}\cap L^1(\R) \times H^{s+1}\cap BV(\R).
\eq
Then the system \eqref{main_sys}-\eqref{ini} and \eqref{def_neu} admits a classical solution $(\rho(t,\cdot),u(t,\cdot))$ such that
\bq\label{est_l1}
\sup_{0 \leq t < T}\lt(\|(\rho - c)(t)\|_{L^1} + \|(\pa_x u)(t)\|_{L^1}\rt) < \infty
\eq
for some $T>0$.
\end{theorem}
In view of the $\dot H^{-1}$ regularity of $\rho_0-c_0$ secured in Theorem \ref{thm_apriori} and its integrability in Theorem \ref{thm_neu_apri}, there  follows the  neutrality condition \eqref{def_neu}, $\displaystyle 
\intr (\rho(t,x)-c(t,x))\, \dx =0, \ \forall t\in[0,T)$.
Thus, Theorem  \ref{thm_neu_apri} identifies the set of initial data \eqref{eq:ICb}, 
which give rise to  \emph{classical solution}  $(\rho,u)$ of the  extended system \eqref{main_sys}, \eqref{def_neu}. Again, Theorem  \ref{thm_neu_apri} emphasizes that the density $\rho(t,\cdot)$ cannot be integrable; only $(\rho-c)(t,\cdot)$ is.

\smallskip
In our next Theorem \ref{anomal}, we construct an \textit{anomalous solution}, exhibiting the non-propagation of the neutrality in time even with the neutralized initial data. To circumvent this anomaly, additional regularity must be imposed, for instance, $u_0 \in BV(\R)$. Consult Lemma \ref{intgr} for the discussion on the constant background case, where the necessary and sufficient condition for the propagation of neutrality is presented.
\begin{theorem} 
    [{\bf Non-propagation of neutrality: Anomalous solution}]\label{anomal} Consider the EP system \eqref{main_sys} with $c = \bar{c} > 0$. There exists a initial data $(\rho_0,u_0)$ satisfying 
\[
(\rho_0 -\bar{c},u_0) \in H^s\cap \dot{H}^{-1}\cap L^1(\R)\times H^{s+1}(\R),
\]
such that any solution $(\rho,u)$ of class $\calC^1$\footnote{In this paper, we say $f$ is of class $\calC^1$ if $f \in \calC^1\lt([0,T)\times\R\rt)$ i.e., $|f(t,x)| + |\pa_tf(t,x)|+ |\pa_xf(t,x)|$ is uniformly bounded in $[0,T']\times\R$ for any $T'<T$. We note that any solution constructed in Theorem \ref{thm_apriori} is of class $\calC^1$ by means of Morrey's embedding.} to \eqref{main_sys}--\eqref{ini} with initial data $(\rho_0,u_0)$ does not fulfill the neutrality condition $\eqref{def_neu}$.
\end{theorem}

\subsubsection{\textbf{The case of attractive forcing}}

In this part, we discuss how the neutrality condition affects the critical threshold phenomenon of the attractive EP system, which is considered to be irrelevant. The threshold conditions for the attractive EP system were derived earlier in \cite{ELT01, CCZ16, BL20_2}. The sub-critical condition for the existence of a global classical solution of the attractive EP system is given by
\bq\label{eq:var-CT}
    \lt( \frac{ -\nu + \sqrt{\nu^2 +4 c_{-}}}{2} \rt) u_0'(x)  \geq (\rho_0(x)-c_-) \qquad \forall x \in \R,
\eq
where $c_-:=c>0$ for the constant \cite[Theorem 3.2]{ELT01}, \cite[Theorem 2.5]{BL20_2} and $c_-:=\min c(\cdot,\cdot)>0$ for the variable background case \cite[Theorem 2.2]{BL20_2}. In particular, this threshold is sharp for the constant background case; if there exists $x\in\R$ which fails to satisfy the inequality in \eqref{eq:var-CT}, then the solution forms singularity in a finite time. In the following theorem, however, we find that the sub-critical region should be further restricted to its borderline zone ---  initial configurations satisfying \emph{equality} in \eqref{eq:var-CT}, in order for the neutrality condition to be fulfilled.

\begin{theorem}[{\bf Reduction to the borderline}]\label{thm_a_con} Suppose that $(\rho,u)$ is a global-in-time classical solution of the attractive EP system \eqref{main_sys}--\eqref{ini}--\eqref{def_neu}, subject to initial data satisfying \eqref{eq:var-CT}.
Then, the set of admissible $c(\cdot,\cdot)$'s is necessarily restricted to the constant background, namely ---  $c(\cdot,\cdot) = c_{-}$ and  \eqref{eq:var-CT} is reduced the borderline case
\begin{equation}\label{att:cri:neu}
    \lt( \frac{ -\nu + \sqrt{\nu^2 +4 c_{-}}}{2} \rt) u_0'(x)  =  (\rho_0(x)-c_-) \qquad \forall x \in \R.
\end{equation}
\end{theorem}

\begin{remark}
Theorem \ref{thm_a_con} shows that in order to entertain the neutrality condition, the sub-critical criteria for variable backgrounds, \eqref{eq:var-CT}, is too restrictive in the sense that it will rule out all but the borderline cases. In particular, for the constant background case $c=c_-$,  the solution $(\rho, u)$ of neurality-satisfying attractive EP system has global-in-time regularity if and only if its initial data $(\rho_0, u_0)$ satisfies \eqref{att:cri:neu}.
\end{remark}

This reduction to the borderline is a main theme in the attractive case. We demonstrate this in the context of variable background, $c=c(t,x)$. Now assume that $|u_0(x)|\stackrel{|x|\rightarrow \infty}{\longrightarrow}0$.
 \sout{Multiply \eqref{att_sub} by $\lt(\frac{2 c_{-}}{ \nu + \sqrt{\nu^2 +4 c_{-}}} \rt)$ and} We integrate \eqref{eq:var-CT} to find
\bq\label{att_sub-x}
\lt( \frac{ -\nu + \sqrt{\nu^2 +4 c_{-}}}{2}\rt) u_0(x){\Big|}^R_{-R} \ge  \int \limits_{-R}^R \big(\rho_0(x)-c_0(x)\big)\dx + \int \limits_{-R}^R \big(c_0(x)-c_-\big)\dx. 
\eq
By assumption, $u_0(x)\Big|^R_{-R}$  vanishes as $R\uparrow \infty$ (in fact, equals zero on the periodic case) and the neutrality condition yields $\displaystyle \int_{\R}  \big(c_0(x)-c_-\big)\dx = 0$, which implies  $c_0(x)=c_{-}$ and we get equality at \eqref{att_sub-x}${}_{R\rightarrow \infty}$. By assumption, both sides tend to zero with $R\uparrow \infty$, hence the equality \eqref{eq:var-CT} takes place almost everywhere. Indeed, for initial data satisfying \eqref{eq:var-CT}, we have
\[
\lt( \frac{ -\nu + \sqrt{\nu^2 +4 c_{-}}}{2} \rt) \pa_x u(t,x)  \geq (\rho(t,x)-c_-) \qquad \forall x \in \R,
\]
so that we can repeat the previous argument for each $t>0$; in particular, we get $c(\cdot,\cdot) = c_{-}$. In Section \ref{sec_att}, we recover the critical threshold results for attractive EP systems with constant/variable backgrounds via the Lyapunov function method. It captures the solution trajectories in a geometric way so that it can be more intuitively understood than the previous approaches. More importantly, a notable feature of the Lyapunov-based approach is that it goes beyond the attractive forcing case; it can be employed in obtaining the critical thresholds for the repulsive EP systems, in particular, with variable backgrounds.

\subsubsection{\textbf{The case of repulsive forcing}}\label{ssec:rep}

We now turn our attention to the study of 
critical thresholds for the repulsive EP systems. Unlike the attractive interaction case, there is no literature on the sub-critical criteria for the repulsive EP system in the variable backgrounds. We provide the first sub-critical results in this case. In order to present our results more precisely, we must introduce the method of phase plane analysis. To this end, we consider the auxiliary variables $(s,w)$ (introduced by H. Liu, S. Engelberg and the last author \cite{ELT01} in their study of critical thresholds  of EP system  with constant backgrounds),
\[
s := \frac{1}{\rho}, \quad \mbox{and}\quad w := \frac{\pa_x u}{\rho}.
\]
Let $\square':= \dot{\square}(t,x(t))$ denote differentiation along particle path 
$x(t)=x(t; \alpha)$,
\begin{equation}
    \label{eq:chr}
\dot{x}(t;\alpha)=u(t,x(t; \alpha)), \ x(0;\alpha)=\alpha.
\end{equation}
We find that (abusing notations) $w(t)=w(t,x(t))$ and $s(t)=s(t,x(t))$ satisfy
\begin{equation}\label{reduced_eq}
\lt\{\begin{split}
w' &= -\nu w + k(1 - c s) \\
s' &= w. 
\end{split}\rt.
\end{equation}

We observe that the question of global-in-time regularity vs. finite-time breakdown is reduced to whether $s(t)$ attains zero --- $\rho(t,\cdot)$ blows up ---  in a finite time. There are two classes of initial configurations $(w_0,s_0)$ of our interest. We say $\Sigma^\flat$ is a sub-critical region of \eqref{reduced_eq} if it is an \emph{invariant} set, i.e., $(w_0,s_0)\in \Sigma^\flat$ implies $(w(t),s(t)) \in \Sigma^\flat$ for all $t>0$. On the other hand, we say $\Sigma^\sharp$ is a super-critical region of \eqref{reduced_eq} if for each $(w_0,s_0)\in \Sigma^\sharp$, there exists $T_*>0$ such that $s(T_*)=0$. Hence, the set of sub/super-critical region gives rise to the criteria of initial data $(\rho_0,u_0)$ for which the local-in-time classical solution of \eqref{main_sys}--\eqref{ini} becomes global or forms singularity in a finite time. Lastly, we say the critical threshold is \emph{sharp} if
\[
\R\times\R_{+} = \Sigma^\flat \cup \Sigma^\sharp.
\]

For the constant backgrounds, where the solutions to \eqref{reduced_eq} can be obtained explicitly, the sharp critical threshold was obtained in \cite{ELT01} for the undamped case, precisely given as:
\bq\label{CT:rep:con}
\Sigma^\flat=\lt\{-\sqrt{s(2-\bar c s)}<w<\sqrt{s(2-\bar c s)}\rt\}, \ \Sigma^\sharp = \R\times\R_+\setminus \Sigma^\flat,
\eq
where $c = \bar c >0$ denotes the background. Recently, Bhatnagar and Liu obtained the sharp critical thresholds for the damped case \cite{BL20}.

For the variable backgrounds, however, we note that $c=c(t,x(t,\alpha))$ in \eqref{reduced_eq}${}_2$ varies along the characteristic and the system \eqref{reduced_eq} is not closed. We overcome this non-locality by exploiting the uniform bounds of $c$. We compare the system \eqref{reduced_eq} to the one associated with its extremes $c=c_\pm$, and we establish the critical thresholds for the variable background by the solutions originating from the constant backgrounds. This idea --- a Lyapunov-based approach --- is widely used for attractive and repulsive force in Sections \ref{sec_att} and \ref{sec_rep}. By nature of this comparison method, all our results in the sequel are sharp in the sense that we can recover the sharp critical threshold results for the constant backgrounds once we set $c_-=c_+$.

Throughout this paper, we use
\[
\calL(w,s) = w + f(s)
\]
as a general form of the Lyapunov function. $f$ will be chosen so as to make the \emph{sign} of $\calL$ determine the critical threshold regions. In particular, $\calL(w,s) = 0$ designates the boundary of sub/super-critical regions --- often, we will call this by threshold line. 

To obtain the threshold lines of the (damped) repulsive EP system, we make use of the following Lyapunov functions,
\bq\label{lya}
\calL_P(w,s):= w + \sqrt{2P(s)}, \ \ \calL_N(w,s):= w - \sqrt{2N(s)},
\eq
where $P$ and $N$ are the maximal solutions of the following ODEs
\bq\label{aux_P}
\frac{dP}{ds} = \nu\sqrt{2P(s)} + 1 - {c_1} s, \ \ \ P(0)=0,
\eq
\bq\label{aux_N}
\frac{dN}{ds} = -\nu\sqrt{2N(s)} + 1 - {c_2} s , \ \ \ N(s_*)=0,
\eq
with constants $c_1,c_2>0$, $\nu\ge0$ and $s_* > \frac{1}{c_2}$. The choice of $c_1$ and $c_2$ will depend on the context, either $\bar c$ for the constant background or $c_\pm$ for the variable background case.

To exploit these Lyapunov functions, we characterize the domains of $\sqrt{2P(\cdot)}$ and $\sqrt{2N(\cdot)}$.   By abusing notation, we denote the domain of $\sqrt{2P(\cdot)}$ by $\textnormal{Dom}(P)$. We appeal the parallel statement to $N$. Indeed, the domains are characterized as
\bq\label{dom}
\textnormal{Dom}(P) = \begin{cases}
[0,\infty) \quad &\nu \ge 2\sqrt{c_1}\\
[0,\tilde{s}] \quad &0 \le \nu < 2\sqrt{c_1}
\end{cases}, \qquad \textnormal{Dom}(N) = \begin{cases}
(-\infty,s_*] \quad &\nu \ge 2\sqrt{c_2}\\
[s_{**},s_*] \quad &0 \le \nu < 2\sqrt{c_2}
\end{cases},
\eq
where
\bq\label{str}
\tilde{s}=\frac{1+e^{\gamma_1}}{c_1}, \ \ \ s_{**} = \frac{1}{c_2} - \lt(s_*-\frac{1}{c_2}\rt)e^{\gamma_2}, \ \ \ \gamma_i:= \frac{\pi\nu}{\sqrt{4c_i-\nu^2}},\ \ \ i\in\{1,2\}.
\eq
In the context of constructing the sub/super-critical regions, we always choose $s_*:= \tilde{s}$ so as to make these regions as large as possible. Here, we mention that the formula of $s_{**}$ plays a vital role in establishing the critical regions.

\noindent

The main novelty of this approach originated with M. Bhatnagar \& H. Liu in \cite{BL20},  with the purpose of replacing an explicit \emph{algebraic description} of the critical threshold, by a \emph{differential description} expressed in terms of a maximal solution of the ODE in $w$--$s$ plane:
\bq\label{aux_Q}
\displaystyle \frac{d Q(s)}{ds} = \nu + \frac{1 -  \bar c s}{Q(s)}.
\eq
Indeed, while $\sqrt{2P}$ and $-\sqrt{2N}$ coincide with  the maximal solution $Q$ of \cite{BL20} in the case of  $c_1=c_2=\bar c$, here we make a distinction between $\sqrt{2P}$ and $-\sqrt{2N}$, as we trace the the \emph{sign} of $Q$, corresponding to  the positive and negative solution.
In Section \ref{sec_rep}, we provide precise characterization of domains \eqref{dom} -- \eqref{str}, replacing the  graphical representation associated with \eqref{aux_Q} in \cite{BL20}.

\vspace{0.2cm}
\noindent
{\bf Comparison Principle.} We now introduce our comparison principles, deploying the Lyapunov functions \eqref{lya}, to establish the critical thresholds. For this, we denote $P_\pm$ by the solution of \eqref{aux_P} with $c_1=c_\pm $ and $N_\pm$ by that of \eqref{aux_N} with $c_2 = c_\pm$.   The \emph{weak} comparison principle is given as follows.
\bq\label{CP:sup}
\begin{split}
\calL_{P_-}(w(t),s(t)) \le 0 \ \ \ &\mbox{if} \ \ \ \calL_{P_-}(w_0,s_0) \le 0, \\
\calL_{N_+}(w(t),s(t)) \ge 0 \ \ \ &\mbox{if} \ \ \ \calL_{N_+}(w_0,s_0) \ge 0
\end{split}
\eq
provided that 
\[
\forall \tau\in[0,t], \ \ \  s(\tau) \in \textnormal{Dom}(P_-) \ \  \lt(\textnormal{resp.} \ \ s(\tau) \ge 0 \ \ \textnormal{and} \ \ s(\tau) \in \textnormal{Dom}(N_+)\rt).
\]
This gives rise to the construction of super-critical regions.
\[
\Sigma^{\sharp}:= \R\times\R_+ \setminus \lt\{-\sqrt{2P_-(s)}<w\rt\} \ \ \textnormal{for} \ \ \nu\ge2\sqrt{c_-}, \ \mbox{and}
\]
\bq\label{CT:rep:sup:wd}
\Sigma^\sharp:=\R\times\R_+ \setminus \lt\{-\sqrt{2P_-(s)}<w<\sqrt{2N_+(s)}\rt\} \ \ \textnormal{for} \ \ 0\le\nu<2\sqrt{c_-}.
\eq
On the other hand, we have \emph{strong} comparison principle. 
\bq\label{CP:sub}
\begin{split}
\calL_{P_+}(w(t),s(t)) > 0 \ \ \ &\mbox{if} \ \ \ \calL_{P_+}(w_0,s_0) > 0, \\
\calL_{N_-}(w(t),s(t)) < 0 \ \ \ &\mbox{if} \ \ \ \calL_{N_-}(w_0,s_0) < 0
\end{split}
\eq
as long as
\[
\forall \tau\in[0,t], \ \ \  s(\tau) \in \textnormal{Dom}(P_+) \ \  \lt(\textnormal{resp.} \ \ s(\tau) \ge 0 \ \ \textnormal{and} \ \ s(\tau) \in \textnormal{Dom}(N_-)\rt).
\]
It leads  us to define 
\[
\Sigma^\flat:= \lt\{-\sqrt{2P_+(s)}<w\rt\}  \ \ \textnormal{for} \ \ \nu\ge2\sqrt{c_+}, \ \mbox{and}
\]
\bq\label{CT:rep:sub:wd}
\Sigma^\flat:= \lt\{-\sqrt{2P_+(s)}<w<\sqrt{2N_-(s)} \rt\}  \ \ \textnormal{for} \ \ 0\le\nu<2\sqrt{c_+}.
\eq
It is not always the case that the set $\Sigma^\flat$ \eqref{CT:rep:sub:wd} is sub-critical; some extra conditions should be met to secure the sub-criticality.

\vspace{0.2cm}
\noindent
{\bf Closing Condition.} 
In order to prove that $\Sigma^\flat$  \eqref{CT:rep:sub:wd} and $\Sigma^\sharp$ \eqref{CT:rep:sup:wd} are \emph{indeed} sub/super-critical regions, we require 
\bq\label{closing}
s_{**} \le 0
\eq
to hold. We call \eqref{closing} by \emph{closing condition}. The term \emph{closing} is motivated by the geometric depiction that the $\Sigma^\flat$ and $\R\times\R_+ \setminus \Sigma^\sharp$ are enclosed by its boundaries as in Figures \ref{fig:rep:sup:ud} and \ref{fig:CT:all}.

The closing condition plays a key role in constructing a sub/super-critical region. We illustrate this by demonstrating that $\Sigma^\sharp$ in \eqref{CT:rep:sup:wd} is a super-critical set. We readily check that the closing condition is validated for $\Sigma^\sharp$. Indeed, by assigning $(c_1,c_2)=(c_-,c_+)$ to \eqref{str}, we obtain 
\[
s_* \ge \frac{2}{c_-}>\frac{1}{c_+},\ \ \ \mbox{and}\ \ \ s_{**} \le \frac{1}{c_+}- \lt(\frac{2}{c_-}- \frac{1}{c_+}\rt) \le 0.
\]
In particular, it implies that $\calL_{N_+}(w,s)=0$ line crosses $w$-axis at $\sqrt{2N_+(0)}\ge0$ so that the boundary of $\Sigma^\sharp$ can be decomposed as:
\[
\{ \calL_{P_-}(w,s)=0 \} \cup \{ \calL_{N_+}(w,s)=0, s\ge0\} \cup \{w\ge \sqrt{2N_+(0)},s=0\} \cup \{w\le 0, s=0\}.
\]

By weak comparison principle \eqref{CP:sup}, we deduce that if $(w_0,s_0)\in \Sigma^\sharp$, then $(w(t),s(t))$ can escape $\Sigma^\sharp$ only through $s=0$ line (with $w\le0$ because $s(t)'=w(t)\ge \sqrt{2N_+(0)}>0$ on the right part of horizontal boundary.) In fact, for each $(w_0,s_0)\in \Sigma^\sharp$, there exists $t_*<+\infty$ such that $s(t_*)=0$. Consult  Proposition \ref{rep_blo_1},  \ref{rep_blo_2}. It proves that $\Sigma^\sharp$ is a super-critical set. Hence, we obtain a super-critical condition for the undamped repulsive EP with variable backgrounds.  

\begin{theorem}[{\bf Finite-time breakdown}] 
\label{thm:rep:sup}Consider the repulsive EP system \eqref{main_sys}--\eqref{ini}. For a given initial data $(\rho_0, u_0)$, if there exists $x\in\R$, which does not satisfy:
 \begin{align*}
&\bullet &  \qquad \ 
 &-\rho_0(x)\sqrt{2P_-\lt(\frac{1}{\rho_0(x)}\rt)} < u_0'(x)   \  &\textnormal{when}& \ \   \nu \geq 2\sqrt{\cmin};\\
& \bullet  &  \ \  
 &-\rho_0(x)\sqrt{2P_-\lt(\frac{1}{\rho_0(x)}\rt)} < u_0'(x)  < \rho_0(x)\sqrt{2N_+\lt(\frac{1}{\rho_0(x)}\rt)}   \ &\textnormal{when}& \ \ 0 \leq \nu < 2\sqrt{\cmin},
\end{align*}
then the solution $(\rho, u)$ will lose $\calC^1$ regularity in a finite time. 
\end{theorem}

\begin{remark}
For the undamped $\nu=0$ case, we have
\bq\label{CT:rep:sup:ud}
\Sigma^\sharp= \R\times\R_+ \setminus \lt\{- \sqrt{s(2-c_-s)} < w < \sqrt{s(2-c_{+}s)+\frac{4}{c_{-}}\lt(\frac{c_{+}}{c_{-}}-1\rt)}\rt\}.  
\eq
Note that $\Sigma^\sharp$ coincides with \eqref{CT:rep:con} if we set $c_-=c_+=\bar c$. Moreover, it improves previously known super-critical results for repulsive EP systems with variable backgrounds \cite{BCKpre, L06}. In the phase plane coordinates, these previous super-criteria can be expressed as $s \ge  \frac2{c_{-}}$ and $w+\sqrt{2s} \le 0$, respectively. In fact, \eqref{CT:rep:sup:ud} contains these two regions. We display this in Figure \ref{fig:rep:sup:ud} for the reader's convenience.
\end{remark}

\begin{figure}[h!]
\hspace{-1.2cm}
\begin{subfigure}[]{0.4\textwidth}
    \includegraphics[scale=.4]{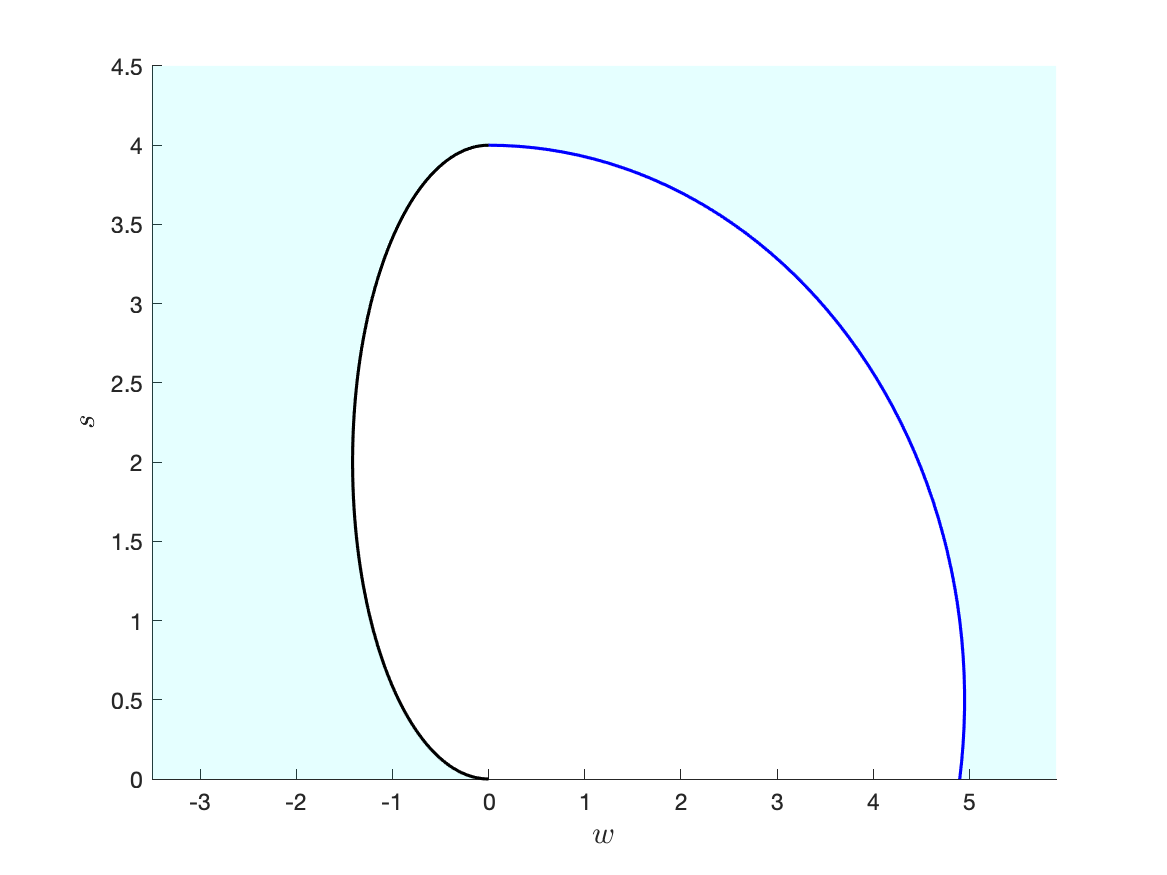}
    \caption{The super-critical set in  \eqref{CT:rep:sup:ud}}
\end{subfigure}
\hspace{1.2cm}
\begin{subfigure}[]{0.4\textwidth}
    \includegraphics[scale=.4]{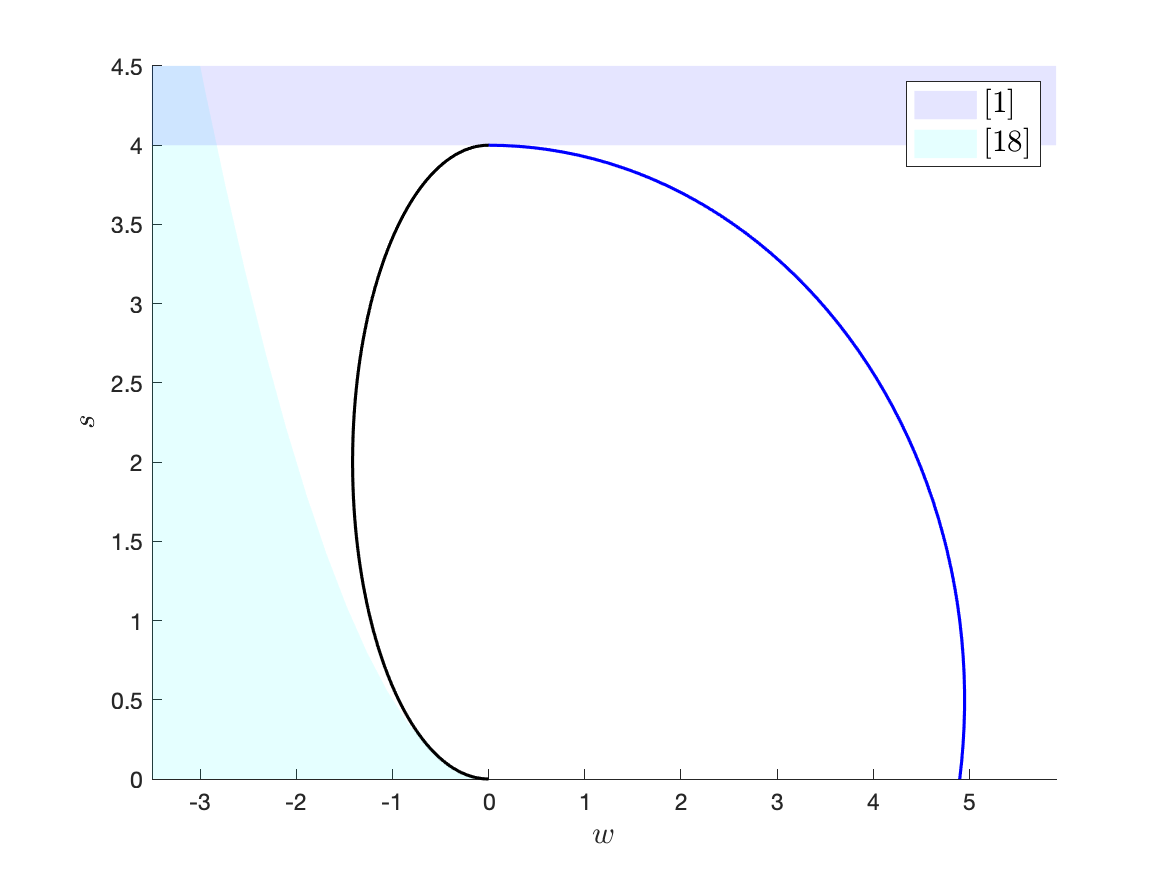}
    \caption{{\color{violet}Violet}: $s\ge 2/{c_-}$, {\color{cyan} Blue}: $w+\sqrt{2s}\le 0$}
\end{subfigure} 
\caption{Ilustration of the super-critical regions}
	\label{fig:rep:sup:ud}
\end{figure}

\vspace{0.2cm}

We now turn our attention to the sub-critical case. The closing condition can also be exploited to show that $\Sigma^\flat$ in \eqref{CT:rep:sub:wd} is a sub-critical set. Once we assume the closing condition hold for $\Sigma^\flat$, then the boundary of $\Sigma^\flat$ can be decomposed as
\bq\label{bdy:sub}
\lt\{\calL_{P_+}(w,s)=0 \rt\} \cup \lt\{\calL_{N_-}(w,s)=0, s > 0\rt\} \cup \lt\{0 <  w \le \sqrt{2N_-(0)}, s=0 \rt\}.
\eq
Note that $\Sigma^\flat$ is an open set, and we can prove that any trajectory issued inside this set cannot hit the boundary by strong comparison principle. Consult Proposition \ref{rep_glo_2} for the details.

Unlike the super-critical set, the closing condition is not satisfied for $\Sigma^\flat$ in \eqref{CT:rep:sub:wd} in general. For instance, if $\nu =0$ with $c_-<c_+$, then the closing condition is violated because
\[
s_{**} = \frac{2}{c_-}-\frac{2}{c_+}>0.
\]

In fact, the closing condition is not merely an artifact of technique but reflects the difficulties in obtaining sub-critical regions for the repulsive EP system. To exhibit this, we introduce the following system of ODEs.

\begin{equation}\label{test_sys}
\begin{cases}
w' = 1-(1+\epsilon \sin(t))s, \\
s' = w, \\
\end{cases}
\end{equation}
with $(w_0,s_0) = (0,1)$. Although the deviation of \emph{background} from $1$ is assumed to be small as $|(1+\epsilon \sin(t))-1|\le \epsilon$, the oscillation of $(w(t),s(t))$ and $c(t)$ can be aligned so that it makes a \emph{resonance}. In particular, $s(t)$ attains zero in a finite time. We implemented the numerical experiment of \eqref{test_sys} with $\epsilon=0.05$ using the built-in ode45 Matlab command. See Figure \ref{fig_gl_rmk} for its depiction. This example convinces us that the closing condition is necessary to construct a sub-critical region.

\begin{figure}[h]
\begin{center}
\includegraphics[scale=.4 ]{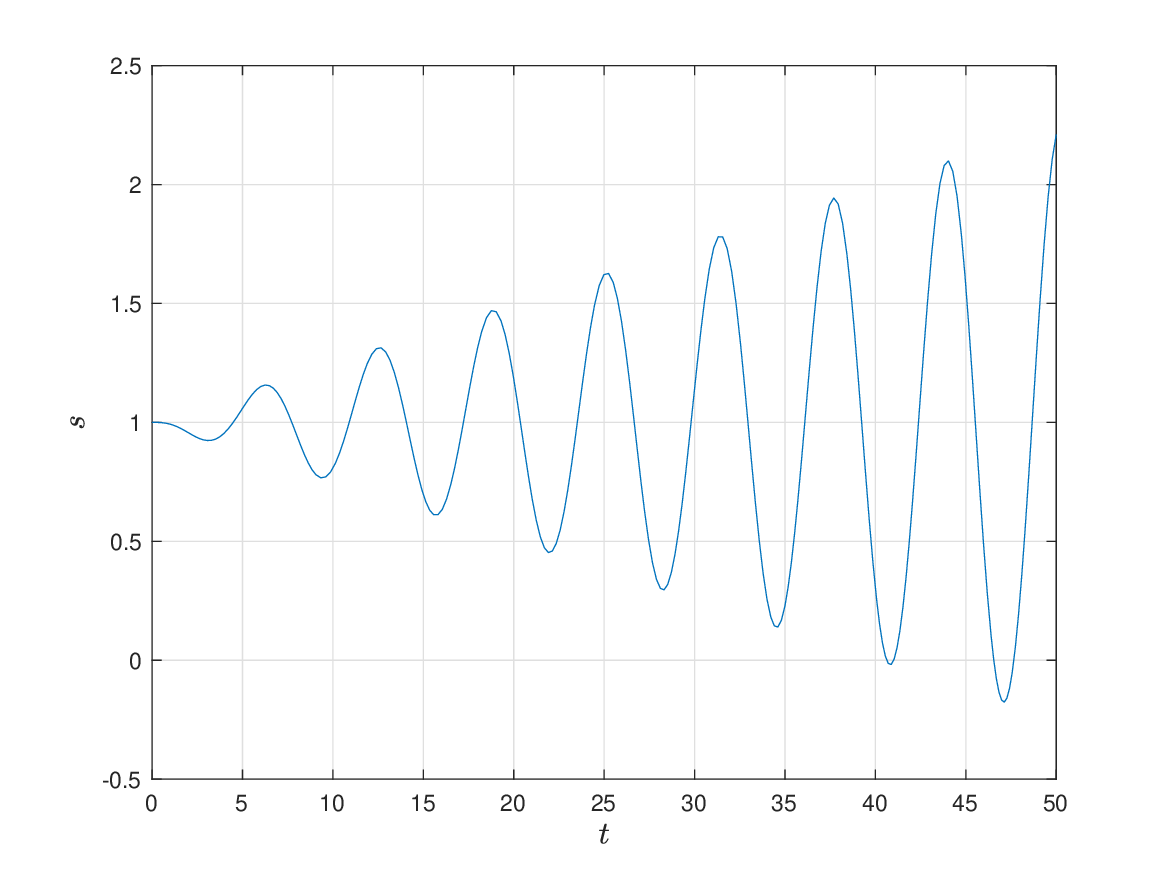}
\quad
\includegraphics[scale=.4 ]{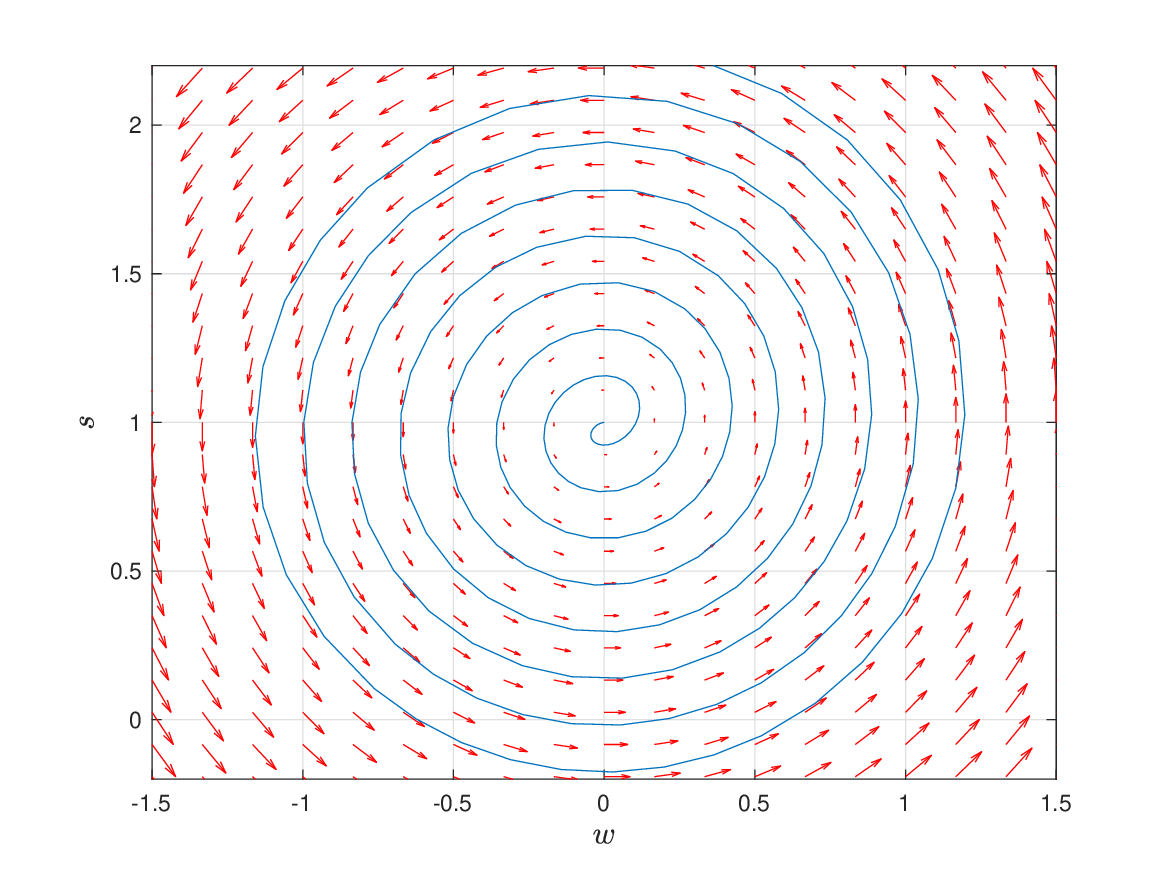}
\caption{The evolution of $s$ (left) and the phase plane plot (right) of the solution trajectory of \eqref{test_sys}.}
\label{fig_gl_rmk}
\end{center}
\end{figure}

In the presence of damping ($\nu > 0$), the closing condition always holds for $\nu\ge2\sqrt{c_+}$ case. For $0<\nu < 2\sqrt{c_+}$ case, the closing condition can be entertained by imposing suitable conditions on $c_+$, $c_-$, and $\nu$. The heuristic is that the ratio of $c_+$ over $c_-$ is close enough to $1$, and the damping effect is large enough to control the oscillation so that the system is similar to that of a constant background case.
 
\begin{theorem}[{\bf Global-in-time regularity}]\label{thm:rep:sub}Consider the repulsive EP system \eqref{main_sys}--\eqref{ini}.  It admits a global classical solution, depending on the relative size of $\nu$ and $\sqrt{c_\pm}$. Specifically --- global classical solution exists if the initial data $(\rho_0, u_0)$ falls into one of the following cases.\newline 
$\bullet$ \ \ Case \textnormal{\#1}. \ \  $\nu \geq 2\sqrt{\cmax}$ \ and \ 
$\displaystyle 
-\rho_0(x)\sqrt{2P_+\lt(\frac{1}{\rho_0(x)}\rt)} < u_0'(x)$;\newline
$\bullet$ \ \ Case \textnormal{\#2}. \ \ $0<\nu < 2\sqrt{\cmax}$ and 
\[
-\rho_0(x)\sqrt{2P_+\lt(\frac{1}{\rho_0(x)}\rt)} < u_0'(x)  < \rho_0(x)\sqrt{2N_-\lt(\frac{1}{\rho_0(x)}\rt)},  \ \ s_+:=\frac{1+e^{\gmax}}{\cmax}, \  \gamma_\pm:= \frac{\pi\nu}{\sqrt{4c_\pm-\nu^2}},
\]
for all $x\in\R$ and one of the following two sub-cases is fulfilled:\newline
\bq\label{ext_dam_11-22}
\left\{ \ \ \begin{split}
\bullet \hspace*{0.5cm} {Case \ \#2.1}:  \quad   & \textnormal{either} \ \ \nu \geq 2\sqrt{\cmin}  \ \ \textnormal{and} \  \ s_+\cmin >1,   \hspace*{5cm}\\
\bullet \hspace*{0.5cm}{Case \ \#2.2}: \quad  &  \textnormal{or} \ \ 0<\nu < 2\sqrt{\cmin}  \ \ \textnormal{and} \  \ e^{\gmin}\lt(s_+\cmin-1\rt) \geq 1.
\end{split}\right.
\eq
\end{theorem}

\begin{remark}
In fact, the condition \eqref{ext_dam_11-22} is equivalent to the closing condition \eqref{closing}. We arrive at  \eqref{ext_dam_11-22} is by putting $(c_1,c_2)=(c_+,c_-)$ in \eqref{str}. If we assume $c_-=c_+$, the closing condition is satisfied without any extra condition, and we recover the sharp critical threshold result for the damped EP system in \cite{BL20}.
\end{remark}

We display the critical threshold results for the repulsive interaction case in Figure \ref{fig:CT:all}.

\begin{figure}[h]
\hspace{-1.2cm}
\begin{subfigure}[]{0.3\textwidth}
    \includegraphics[scale=.3]{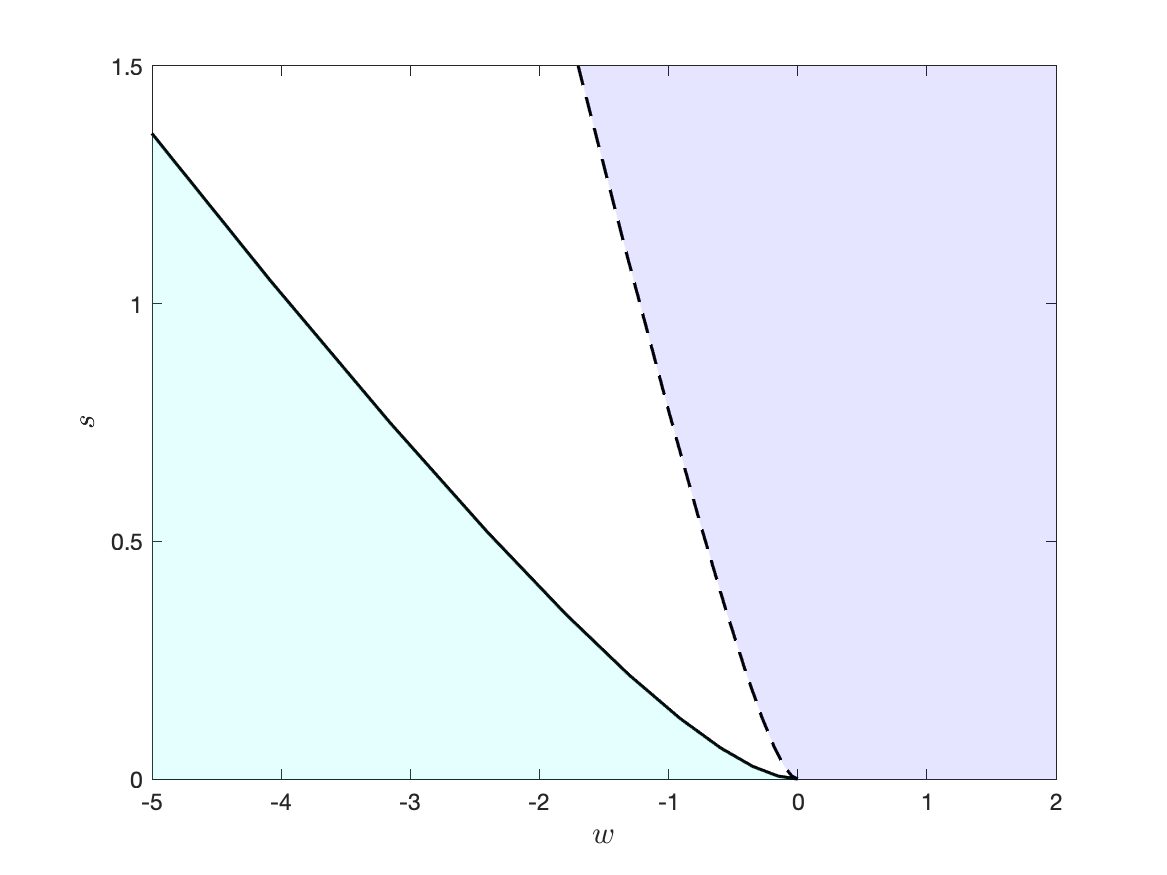}
    \caption{$\nu > 2 \sqrt{ c_{+}}$ case 
}
\end{subfigure}
\hspace{0.3cm}
\begin{subfigure}[]{0.3\textwidth}
    \includegraphics[scale=.3]{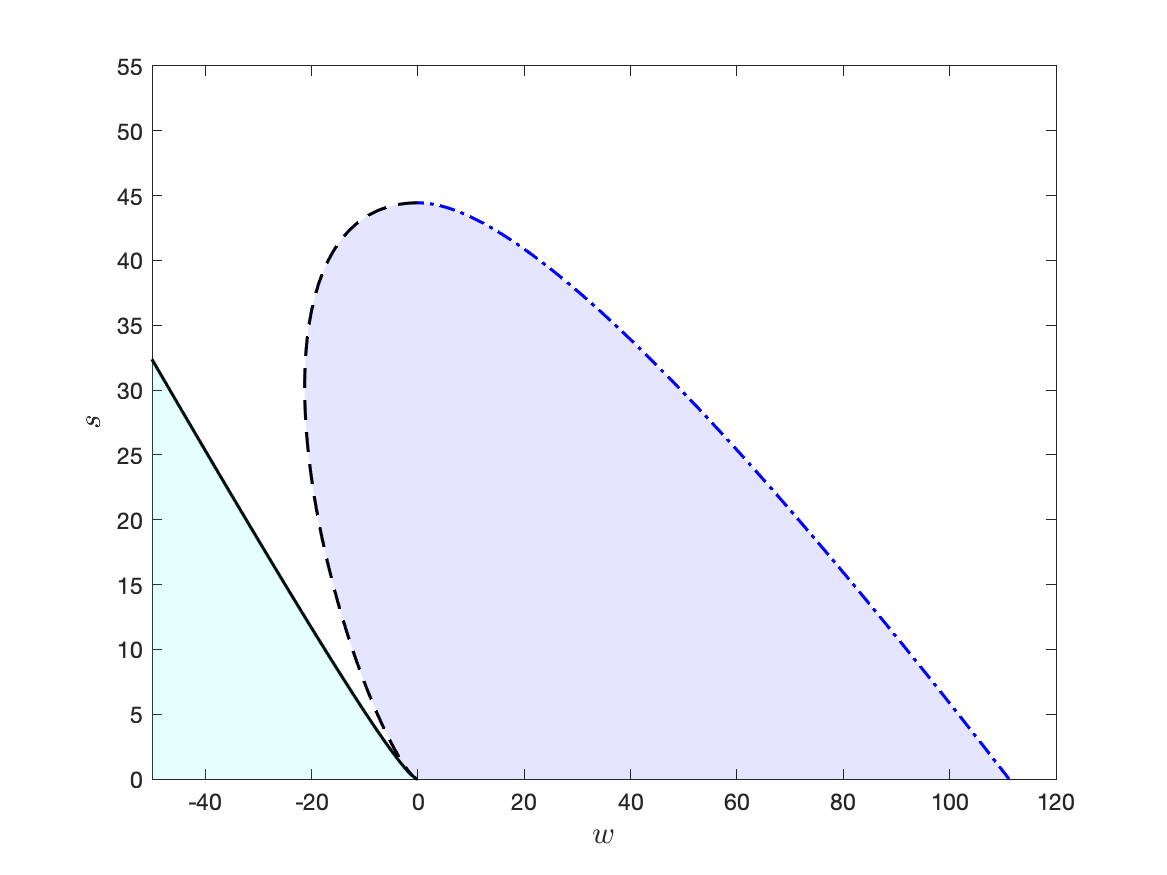}
    \caption{$2 \sqrt{c_{-}} \le \nu < 2 \sqrt{c_{+}}$ case }
\end{subfigure} 
\hspace{0.3cm}
\begin{subfigure}[]{0.3\textwidth}
    \includegraphics[scale=.3]{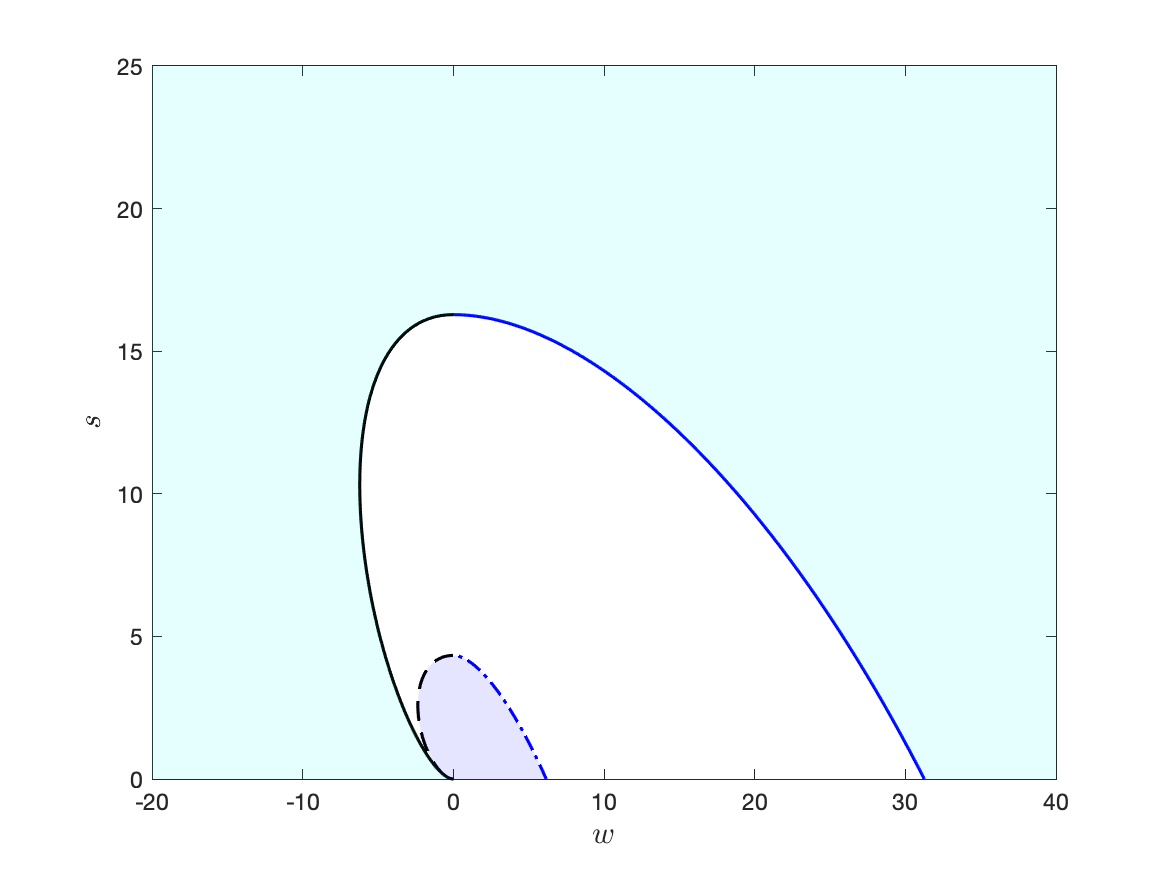}
    \caption{$  0 < \nu < 2 \sqrt{ c_{-}}$ case}
\end{subfigure}
        \caption{Illustration of the \textcolor{violet}{sub}/\textcolor{cyan}{super}-critical regions for the repulsive EP system}
\label{fig:CT:all}
\end{figure}

%
%
%
%
%
%
%

\subsubsection{\textbf{Application to the EP system for cold plasma ion dynamics}}

Our final result is finite-time singularity formation and the global regularity of the damped EP system for cold plasma ion dynamics. Consider a plasma consisting of electrons and ions on $\R$, where both electrons and ions have constant temperature. We assume that the temperatures of electrons and ions equal one and zero, respectively. In this case, our main system is given by
\begin{equation}\label{main_sys_ion}
\begin{cases}
\partial_t \rho + \partial_x(\rho u) = 0, \\
\partial_t u + u \partial_x u =  -\nu u  - \partial_x \phi, \\
- \partial_{xx}\phi = \rho-e^{\phi},
\end{cases}
\end{equation}
where the density of electrons $\rho_e$ is given by {\it Maxwell--Boltzmann relation} \cite{C84}, i.e., $\rho_e =e^{\phi}$, where $\phi = \phi(t,x)$ is the electric potential. 

Note that the system \eqref{main_sys_ion} corresponds to  \eqref{main_sys} with $(k,c) = (1,e^\phi)$.  We consider the system \eqref{main_sys_ion} around a constant state $(1,0)$, i.e.
\begin{equation}\label{bdy}
(\rho, u)(t,x) \to (1,0) \quad \mbox{as }  |x| \to \infty.
\end{equation}
We notice from the above that $\phi(t,x) \to 0$ as $|x| \to \infty$. For the system \eqref{main_sys_ion}--\eqref{bdy}, the local-in-time existence and uniqueness of smooth solutions are studied in \cite{LLS13}.

In order to apply the framework developed in Section \ref{ssec:rep}, we need to secure lower and upper bounds of $e^\phi$. To this end, we introduce the total energy function $H = H(t)$ for \eqref{main_sys}--\eqref{ini} by
\begin{equation}\label{ci:engy}
H=H(t) := \intr  \left( \frac12 \rho u^2  + \frac12 |\pa_x \phi|^2 + U(\phi)\right) \dx,
\end{equation}
where $U(r):= (r-1)e^r + 1$, which is nonnegative for all $r \in \R$. 

Then, as long as there exist classical solutions, one can readily check that the total energy is not increasing in time, i.e., $H(t) \leq H_0:=H(0)$ for all $t \geq 0$, and conserved in time when $\nu = 0$. It was investigated in \cite{BCKpre} that a uniform bound of the potential function, $\phi$, can be obtained by $H_0$. Hence, we can apply the results in the previous section to this case.

\begin{theorem}[{\bf Finite-time breakdown}]
Consider the repulsive EP system \eqref{main_sys_ion}--\eqref{bdy} with $\nu = 0$. For a given initial data $(\rho_0, u_0)$, if there exists $x\in \R$ which does not satisfy
\begin{equation*}
     - \sqrt{2\rho_0(x)-c_-}< u_0'(x) <\sqrt{2\rho_0(x)-c_++\frac{4}{c_-}\lt(\frac{c_+}{c_-}-1\rt){\rho_0(x)}^2 },
\end{equation*}
then the classical solution $(\rho, u)$ will lose $\calC^1$ regularity in a finite time. Here, $ c_{-} := \exp\lt(V^{-1}_-(H_0)\rt)$ and $ c_{+} := \exp\lt(V^{-1}_+(H_0)\rt)$ denote positive constants, where $V^{-1}_\pm$ are inverses of $V_\pm$ defined in \eqref{fctV}. 
\end{theorem}

For the damped Euler--Poisson system for a cold plasma, the global existence and uniqueness of $H^s$ solutions are obtained in \cite{LP19} under the smallness assumptions on the initial data $(\rho_0, u_0)$ near the constant equilibrium state $(1,0)$ in $H^s$ norm. However, the following theorem shows that the global-in-time existence of classical solutions to the damped Euler--Poisson system for a cold plasma can be obtained {\it only} under the smallness assumption on the initial energy $H_0$ compared to the strength of damping.

\begin{theorem}[{\bf Global-in-time regularity}]\label{thm:ion:glo} 
For a class of initial data $(\rho_0,u_0)$,  whose initial energy $H_0$ \eqref{ci:engy} is small enough, we assume that
\bq\label{ci:sub}
\nu \ge C\sqrt{H_0},
\eq
for some absolute constant $C>0$ and 
\bq\label{ion:subcri}
\|\rho_0-1\|_{L^\infty} + \|u_0' \|_{L^\infty} \le \Lambda_0,
\eq
where $\Lambda_0>0$ is a small absolute constant.
Then the damped cold-ion system \eqref{main_sys_ion}--\eqref{bdy} admits global classical solution with initial data $(\rho_0,u_0)$.
\end{theorem}
\begin{remark}
We note that \eqref{ion:subcri} restricts $(\rho_0,u_0)$ to fall into the sub-critical regions in Theorem \ref{thm:rep:sub}. We can take $\Lambda_0=\frac{1}{2}$, for instance. Thus, we are only concerned with the smallness of $H_0$ and the requirement on the strength of damping \eqref{ci:sub}.
\end{remark}

%
%
%
%
%

\subsection{Outline of the paper} This paper is organized as follows.  In Section \ref{sec_neu}, we establish the local-in-time well-posedness of the EP system without imposing the neutrality condition. We then prove the non-existence result of Theorem \ref{illposed}. Moreover, we study the well-posedness of the EP system equipped with the neutrality condition and construct the anomalous solutions in Theorem \ref{anomal}. In Section \ref{sec_att}, we present our Lyapunov-based approach to establish the critical thresholds in the attractive case. In Section \ref{sec_rep}, we discuss the repulsive interaction case, completing the proofs of Theorems \ref{thm:rep:sup} and \ref{thm:rep:sub}. As a consequence, we prove Theorem \ref{thm:ion:glo}, which concerns the global regularity of damped cold plasma ion dynamics.

%
%
%
%
%

\section{Local well-posedness and the neutrality condition}
\label{sec_neu}

In this section, we present a rigorous mathematical treatment of the local well-posedness and the neutrality condition of the EP system introduced in Section \ref{ssec:lwp}. We begin by proving the local well-posedness theory, which is \emph{independent} of the neutrality condition.

\subsection{Local well-posedness and Non-existence} 

\begin{proof}
[Proof of Theorem \ref{thm_apriori}]
We only provide a priori estimates of solutions in the desired regularity spaces. The local well-posedness can be obtained by the standard arguments developed for the types of conservation laws; see \cite{BL20, BL20_2, M03} for instances. Indeed, we obtain
\begin{align}\label{est:lwp}
&\frac{d}{dt}\lt( \|\rho - c\|_{H^s \cap \dot H^{-1}} + \|u\|_{H^{s+1}}\rt) \cr
& \quad \leq C\lt(\|\pa_x u\|_{L^\infty} + \|\rho\|_{L^\infty} + \|c\|_{L^\infty} +\|\pa_x c\|_{H^{s}}+1 \rt) \lt(\|\rho - c\|_{H^s \cap \dot H^{-1}} + \|u\|_{H^{s+1}} \rt)\cr
&\qquad  +\|\pa_t c\|_{H^s \cap \dot H^{-1}},
\end{align}
where $C>0$ is an absolute constant, see Appendix \ref{app:lwp} for the details of \eqref{est:lwp}. We apply the Gr\"onwall's lemma to obtain \eqref{est_X}, thereby concluding the proof.
\end{proof}

In fact, the proofs of Theorems \ref{illposed} and \ref{anomal} are the main themes in this section. They restrict our attention to the case of a constant background state. For the sake of simplicity in presenting our results, we only focus on the undamped case. The damped case can be treated in a similar way. Our analysis relies on the characteristic method and, especially, the explicit formulation of Jacobian of the characteristic flow. We briefly explain the main tools and the strategy.

For a $\calC^1$ solution $(\rho,u)$ to EP system \eqref{main_sys}--\eqref{ini} with $c =\bar{c}$, we consider the characteristic flow $x(t,\alpha)$ defined in \eqref{eq:chr}. Write 
\begin{equation*}
\Gamma(t,\alpha):=\partial_{\alpha}x(t,\alpha),
\end{equation*}
signifying the Jacobian of the map $x(t,\cdot):\R\to\R$. We note that $\Gamma(0,\alpha)=1$ and $x(t,\cdot)$ is a diffeomorphism as long as $\Gamma(t,\cdot)>0$. In particular, we have
\bq\label{chr:infty}
\lim_{\alpha \rightarrow \pm \infty} x(t,\alpha) = \pm \infty.
\eq
Taking $\partial_{\alpha}$ to \eqref{eq:chr}, we obtain
\begin{equation}\label{rhopart}
    \partial_{x}u(t,x(t,\alpha)) = \frac{\partial_t \Gamma(t,\alpha)}{\Gamma(t,\alpha)}.
\end{equation}
In view of the mass conservation $\rho' = -\rho \partial_x u$, we find that
\begin{equation}\label{paupart}
\rho(t,x(t,\alpha))=\frac{\rho_0(\alpha)}{\Gamma(t,\alpha)}.
\end{equation}
Finally, we recall the explicit formulas of  $\Gamma(t,\alpha)$ for the undamped ($\nu=0$) case \cite{ELT01}.  They will lay the foundation of the proofs throughout this section.

$\bullet$ Repulsive case:
\begin{equation}\label{gam_1}
\Gamma(t,\alpha)
= 1+ \lt(\frac{\rho_0(\alpha)}{\bar c}-1\rt)(1-\cos(\sqrt{\bar{c}} t))+u_0'(\alpha)\frac{\sin(\sqrt{\bar{c}} t)}{\sqrt{\bar{c}}},
\end{equation}

$\bullet$ Attractive case: 
\begin{equation}\label{gam_2}
\Gamma(t,\alpha)= 1+ \lt(\frac{\rho_0(\alpha)}{\bar c}-1\rt)(1-\cosh(\sqrt{\bar{c}} t))+u_0'(\alpha)\frac{\sinh(\sqrt{\bar{c}} t)}{\sqrt{\bar{c}}}.
\end{equation}

We now prove Theorem \ref{illposed} based on a contradiction argument. Assuming the existence of a solution satisfying \eqref{cls:hs}, subject to ``non-neutral" initial data, i.e., 
\[
\intr (\rho_0(x) - \bar c) \,\dx \neq 0,
\]
the non-neutral effect prevents $u$ from decaying at the far field that $u(t,\cdot) \notin H^{s+1}$ as soon as $t>0$. 
\begin{proof}[Proof of Theorem \ref{illposed}]
We set initial data $(\rho_0,u_0)$ by assigning $u_0=0$ and $\rho_0=\bar c + g(x)$ where
$g(x)\in L^1\cap H^s(\R)$ and
\[
\intr g(x) \, \dx \neq 0.
\]
We stress that $g(x) \notin \dot{H}^{-1}(\R)$; otherwise, it would violate the \emph{non-vanishing integral} condition above. Indeed, there are infinitely many such functions; for instance, we choose $g(x)=(1 + x^2)^{-1}$ in this context.

Suppose, by way of deriving a contradiction, that there exists a solution to EP system \eqref{main_sys}--\eqref{ini} satisfying
\bq\label{cls:hs}
(\rho -\bar c , u ) \in L^{\infty}\lt([0,\delta];  H^s(\R) \times H^{s+1}(\R)\rt)
\eq
for some $\delta>0$. We note that this, in particular, requires $u$ to decay at the far field,
\begin{equation}\label{u_limit}
    \lim_{\alpha \rightarrow \pm \infty}u(t,\alpha) = 0, \ \ \ \forall t \in [0,\delta].
\end{equation}
To investigate the solution behavior at infinity, we observe that $\|\pa_x u\|_{L^{\infty}}$ is uniformly bounded in $t\in[0,\delta]$. Hence, we can employ the forward characteristic $x(t,\cdot)$ defined in \eqref{eq:chr}. For each $t \in [0,\delta]$ and $R>0$, we infer from \eqref{rhopart} that
\begin{align*}
\begin{aligned}
u(t,x(t,R))-u(t,x(t,-R))&=\int_{x(t,-R)}^{x(t,R)}(\partial_xu)(t,x)\,\dx\\&=\int_{R}^{R}(\partial_x u)(t,x(t,\alpha)) \dx(t,\alpha)\\
    &=\int_{-R}^{R} \partial_t\Gamma(t,\alpha)\,\textnormal{d}\alpha.
\end{aligned}
\end{align*}
Hence, an application of the explicit formula  \eqref{gam_1}, for the case of repulsive forcing, supplies that
\[
    u(t,x(t,R))-u(t,x(t,-R)) = \int_{-R}^{R}g(\alpha)\,\textnormal{d}\alpha\cdot \frac{\sin(\sqrt{\bar{c}}t)}{\sqrt{\bar{c}}}.
\]
Then, by taking $R \rightarrow \infty$, we find that, for all small $t>0$, 
\[
\lim_{R \rightarrow +\infty} \lt(u(t,x(t,R)) - u(t,x(t,-R))\rt) \neq 0,
\]
yielding a contradiction to \eqref{u_limit} by means of \eqref{chr:infty}. The attractive forcing case can be tackled in the same manner by recalling \eqref{gam_2}. The conclusion of the theorem is now immediate.
\end{proof}

\subsection{Neutrality condition and Anomalous solution}

We next discuss the local well-posedness encompassing the neutrality condition. We prove that the propagation of neutrality is secured in the regularity space recorded in Theorem \ref{thm_neu_apri}.

\begin{proof}[Proof of Theorem \ref{thm_neu_apri}]
    
We estimate the evolution of $\|(\rho - c)(t)\|_{L^1}$ and $\|(\pa_x u)(t)\|_{L^1}$. Straightforward computations give
\begin{align*}
\frac{d}{dt}\|\rho -c\|_{L^1} &\leq \|\pa_x ((\rho - c)u)\|_{L^1} + \|\pa_t c\|_{L^1} + \|\pa_x (cu)\|_{L^1}\cr
&\leq \|\rho - c\|_{H^1}\|u\|_{H^1} + \|\pa_t c\|_{L^1} + \|c\|_{L^\infty}\|\pa_x u\|_{L^1} + \|\pa_x c\|_{L^2}\|u\|_{L^2}\cr
&\leq |\|\{\rho(t,\cdot),u(t,\cdot)\}|\|_s^2 + \|\pa_t c\|_{L^1} + \|c\|_{L^\infty}\|\pa_x u\|_{L^1} + \|\pa_x c\|_{L^2}^2 +  |\|\{\rho(t,\cdot),u(t,\cdot)\}|\|_s^2
\end{align*}
and
\[
\frac{d}{dt}\|\pa_x u\|_{L^1} \leq \|\pa_x u\|_{L^2}^2 + \|u\|_{L^2}\|\pa_{x}^2 u\|_{L^2} + \|\rho - c\|_{L^1} \leq 2 |\|\{\rho(t,\cdot),u(t,\cdot)\}|\|_s^2 + \|\rho - c\|_{L^1}.
\]
By adding these two inequalities, we obtain
\begin{align*}
&\frac{d}{dt}\lt(\|\rho -c\|_{L^1} + \|\pa_x u\|_{L^1}\rt) \cr
&\quad \leq (1 + \|c\|_{L^\infty})\lt(\|\rho -c\|_{L^1} + \|\pa_x u\|_{L^1}\rt) + 4|\|\{\rho(t,\cdot),u(t,\cdot)\}|\|_s^2 + \|\pa_t c\|_{L^1} + \|\pa_x c\|_{L^2}^2. 
\end{align*}
Applying the Gr\"onwall's lemma to the above yields the desired result \eqref{est_l1}, and the theorem follows.
\end{proof}

For the case of constant background, we identify the necessary and sufficient condition of the neutrality condition. It plays a central role in constructing the \emph{anomalous} solutions. 

\begin{lemma}[{\bf Neutrality condition: constant background}] \label{intgr} Consider the EP system \eqref{main_sys} with $c = \bar{c} > 0$. 
Let $(\rho,u)$ be a $\calC^1$ solution of \eqref{main_sys}--\eqref{ini} in $[0,T) \times \R$ with neutrality satisfying initial data, i.e,
\begin{equation}\label{ini:neu}
    \rho_0-\bar c \in L^1(\R),\text{ with } \int_{\R} (\rho_0(x)-\bar c)\, \dx=0.
\end{equation}
Then, the propagation of neutrality condition \eqref{def_neu} holds if and only if
\begin{equation*}
    u_0 \in  BV(\R),\text{ with } \int_{\R} u'_0(x) \,\dx=0.
\end{equation*}    
\end{lemma}

\begin{proof}[Proof of Lemma \ref{intgr}]
Let $x(t,\alpha)$ be the forward characteristic flow satisfying \eqref{eq:chr}. Note that $\Gamma(t,\cdot)>0$ is well-defined for each $t \in [0,T)$. We claim that
\[
\rho(t,\cdot)-\bar c\in L^1(\R) \quad \mbox{if and only if} \quad \Gamma(t,\cdot) -1 \in L^1(\R). 
\]
In this context, we apply \eqref{paupart} to see that 
$$\begin{aligned}
\int_{x(t,-R)}^{x(t,R)} |\rho(t,x)-\bar c|\,\dx &= \int_{-R}^{R} |\rho(t,x(t,\alpha))-\bar c |\,\textnormal{d}x(t,\alpha) \cr
&= \int_{-R}^{R} \lt|\frac{\rho_0(\alpha)}{\Gamma(t,\alpha)} -\bar c \rt|\Gamma(t,\alpha)\,\textnormal{d}\alpha \\
 &= \int_{-R}^{R} |(\rho_0(\alpha)-\bar c) -\bar c(\Gamma(t,\alpha)-1)|\,\textnormal{d}\alpha. 
\end{aligned}$$
On applying the triangle's inequality with   \eqref{chr:infty}, the claim follows. Hence, we deduce from the explicit formulas of $\Gamma(t,\cdot)-1$ recorded in   \eqref{gam_1} and \eqref{gam_2} that $u_0' \in L^1(\R)$ if and only if $\Gamma(t,\cdot)-1 \in L^1(\R)$, which is equivalent to $\rho(t,\cdot)-\bar c \in L^1(\R)$.

We now assume that $u'_0 \in L^1(\R)$. Similarly, in view of \eqref{paupart}, one sees that
\[
\intr (\rho(t,x)-\bar c)\,\textnormal{d}x  = \intr (\rho_0(\alpha)-\bar c) - \bar c (\Gamma(t,\alpha)-1)\,\textnormal{d}\alpha  = -\bar c\intr (\Gamma(t,\alpha)-1) \,\textnormal{d}\alpha.
\]
By inserting \eqref{gam_1}, \eqref{gam_2} into $\Gamma(t,\cdot)-1$  and integrating over $\R$, we discern that 
\[
\intr (\rho(t,x)-\bar c)\,\textnormal{d}x = 0\quad \mbox{ if and only if } \quad \int_{\R} u_0'(\alpha)\,\textnormal{d}\alpha =0.
\]
This confirms the conclusion of the lemma.
\end{proof}

\begin{remark}
It is worth noting that if $u_0 \notin BV(\R)$, we infer that the local $\calC^1$ solution instantaneously fails the neutrality condition by taking $T>0$ arbitrarily small. Hence, this lemma leads us to find \emph{anomalous} solutions in Theorem \ref{anomal}.
\end{remark}

\begin{proof}[Proof of Theorem \ref{anomal}]
We fix $\rho_0 =\bar c>0$. It satisfies \eqref{ini:neu}, evidently. In view of Theorem \ref{thm_apriori} and Lemma \ref{intgr}, we aim to choose $u_0$ satisfying 
\[
u_0 \in H^{s+1}\setminus BV (\R).
\]
Consider
\[
u_0(x) := \frac{\sin x}{(1+x^2)^{\frac{3}{8}}}.
\] 
Since $u_0 \in H^{s+1}(\R)$, we obtain a local smooth solution $(\rho,u)$ via Theorem \ref{thm_apriori}. In particular, we find by means of Morrey's embedding that this solution is of class $\calC^1$. However, upon noting
\[
u_0'(x) = \frac{\cos x}{(1+x^2)^{\frac{3}{8}}} -\frac{3}{4}\frac{x\sin x}{(1+x^2)^{\frac{11}{8}}} \notin L^1(\R),
\]
we find that $(\rho,u)$ fails the neutrality condition  instantaneously by Lemma \ref{intgr}. This completes the proof.
\end{proof}

We, therefore, conclude from Theorems \ref{thm_neu_apri} and \ref{anomal} that it is necessary to assume additional regularity $u_0 \in BV(\R)$ to secure the propagation of the neutrality condition.


 %
 %
 %
 %
 %
 %
\section{Attractive forcing and the restricted  borderline threshold}\label{sec_att}

In this section, we focus on the critical threshold phenomenon of the attractive EP system \eqref{main_sys}--\eqref{ini}. To apply the method of phase plane analysis introduced in Section \ref{ssec:rep}, we consider the following  ODE system:
\begin{equation}\label{reduced_eq_att}
\lt\{\begin{split}
w' &= -\nu w - (1 - c(t)s), \\
s' &= w,
\end{split}\rt.
\end{equation}
subject to $(w(0),s(0))=(w_0,s_0)$, where $c(t)$ is smooth function satisfying $0<c_- \le c(t) \le c_+$. 

\subsection{Constant background case}

To give a motivation for the Lyapunov function, we initiate the discussion for the constant background case --- $c(t)=\bar c$. In this case, \eqref{reduced_eq_att} can be rewritten in the matrix form as,
\bq\label{sys:att:cst}
\left( \begin{array}{c}
w  \\
\displaystyle s  - \frac1{\bar c}  
\end{array} \right)' = 
\left( \begin{array}{cc}
-\nu & {\bar c} \\
1 & 0
\end{array} \right) 
\left( \begin{array}{c}
w  \\
\displaystyle s - \frac1{\bar c}  
\end{array} \right).
\eq
The dynamics of solution trajectories are determined by the eigenvalues of its coefficient matrix, whose eigenvalues are given by
\[
\lambda_{s} = \frac{-\nu - \sqrt{\nu^2 + 4\bar c}}{2} <0 \quad\mbox{and}\quad \lambda_{u} = \frac{-\nu + \sqrt{\nu^2 + 4\bar c}}{2}>0,
\]
with its corresponding eigenvectors $X_{s} = \displaystyle \left( \begin{array}{c} \lambda_{s}  \\
1 \end{array} \right) $ and
$X_{u} = \displaystyle \left( \begin{array}{c} \lambda_{u}  \\
1 \end{array} \right) $.
In particular, the equilibrium point $(0,1/\bar c)$ of \eqref{sys:att:cst} is unstable; the solution curves are attracted to $(0,1/\bar c)$ along $X_s$ direction while repulsed from $(0,1/\bar c)$ along $X_u$ direction.

This motivates us to define the following Lyapunov functions,
\[
\calL_{s}(w,s):= w - \lambda_{s}\lt(s-\frac{1}{\bar c}\rt) \quad \mbox{and}\quad \calL_{u}(w,s):=  w - \lambda_{u}\lt(s-\frac{1}{\bar c}\rt). 
\]
In particular, the solution $(w(t),s(t))$ of \eqref{sys:att:cst} satisfies 
\[
\frac{d}{dt}\calL_{s}(w(t),s(t)) = \lambda_{u}\calL_{s}(w(t),s(t)) \quad \mbox{and}\quad \frac{d}{dt}\calL_{u}(w(t),s(t)) = \lambda_{s}\calL_{u}(w(t),s(t)),
\]
obtaining
\bq\label{att:exp}
\calL_{s}(w(t),s(t)) = \calL_{s}(w_0,s_0)e^{\lambda_ut} \ \  \mbox{and} \ \ \calL_{u}(w(t),s(t)) = \calL_{u}(w_0,s_0)e^{\lambda_st}. 
\eq

Note that $\calL_{s}(w_0,s_0)\ge0$ implies $\calL_{s}(w(t),s(t))\ge0$ for all $t\ge0$. This, combined with $s'=w$ and that $w$-intercept of the line $\calL_{s}(w,s)=0$ is positive, yields that $\{\calL_{s}(w,s)\ge0, s>0\}$ is \emph{invariant}. On the other hand, if $\calL_{s}(w_0,s_0)<0$, then \eqref{att:exp} implies that
\[
\calL_{s}(w(t),s(t))\to -\infty, \ \ \calL_{u}(w(t),s(t)) \to 0 \ \  \mbox{as}  \ \ t \to \infty.
\]
By noticing
\[
\calL_s(w,s)-\calL_u(w,s) = \lt(-\lambda_s+\lambda_u\rt)\lt(s-\frac{1}{\bar c}\rt),
\]
we conclude that $s(t)$ attains zero in a finite time. This argument recovers the sharp critical threshold for the attractive EP with constant background \cite[Theorem 3.2]{ELT01}, \cite[Theorem 2.5]{BL20_2}.

\sout{
From geometric point of view, $\calL_{u}(w,s)$ (resp. $\calL_{s}(w,s)$) gives $X_s$ (resp. $X_u$) coordinate. This can be understood as follows; for $xy$ plane, the equation of $x$-axis is $y=0$ and this $y$ gives the $y$-coordinate.
By a change of coordinate, we obtain
\[
w = -\frac{\lambda_u\calL_s(w,s) - \lambda_s\calL_u(w,s)}{\lambda_s -\lambda_u},\ \ \ \  s = \frac{1}{\bar c} - \frac{\calL_s(w,s) - \calL_u(w,s)}{\lambda_s -\lambda_u} \ \ \ \forall (w,s) \in \R^2.
\]
\[
w(t) = \frac{\lambda_ue^{\lambda_u t}\calL_s(w_0,s_0) - \lambda_se^{\lambda_s t}\calL_u(w_0,s_0)}{\sqrt{\nu^2+4\bar c}},\ \ \ \  s(t) = \frac{1}{\bar c} + \frac{e^{\lambda_u t}\calL_s(w_0,s_0) - e^{\lambda_s t}\calL_u(w_0,s_0)}{\sqrt{\nu^2+4\bar c}} 
\]
}

\subsection{Variable background case}
To study the critical threshold for the variable background case, we consider the following Lyapunov functions:
\[
\calL_{s}^\pm(w,s):= w - \lambda_{s}^\pm \lt(s-\frac{1}{c_\pm}\rt) \quad \mbox{and}\quad \calL_{u}^-(w,s):= w - \lambda_{u}^- \lt(s-\frac{1}{c_-}\rt),
\]
where 
\[
\lambda_{s}^\pm = \frac{-\nu - \sqrt{\nu^2 + 4 c_\pm}}{2} <0 \quad\mbox{and}\quad \lambda_{u}^\pm = \frac{-\nu + \sqrt{\nu^2 + 4 c_\pm}}{2}>0.
\]
We readily check that
\[
\frac{d}{dt}\calL_s^-(w(t),s(t))
\ge \lambda_u^-\calL_s^-(w(t),s(t)),\quad s(t)\ge0.
\]
Similarly, this implies that $\lt\{\calL_s^-(w,s)\ge 0, s>0\rt\}$ is a sub-critical region.
On the other hand, we have 
\[
\frac{d}{dt}\calL_s^+(w(t),s(t)) \le \lambda_u^+\calL_s^+(w(t),s(t)), \ \ \ \frac{d}{dt}\calL_u^-(w(t),s(t))
\ge \lambda_s^-\calL_u^-(w(t),s(t)), \ \ \ s(t)\ge 0.
\]
By Gr\"onwall's lemma, we observe that
\[
\begin{split}
\lt(-\lambda_s^+ + \lambda_u^-\rt)s(t) &= \calL_s^+(w(t),s(t)) - \calL_u^-(w(t),s(t)) + \frac{1}{\lambda_u^+} - \frac{1}{\lambda_s^-} \le \calL_s^+(w_0,s_0)e^{\lambda_u^+ t} + \calO(1) 
\end{split}
\]
as long as $s(t)$ is non-negative; in particular, if $\calL_s^+(w_0,s_0)<0$, then $s(\cdot)$ attains zero in a finite time. Hence, we obtain the sub/super-critical thresholds for attractive EP systems with variable backgrounds. They coincide with \cite[Theorems 2.2 and 2.3]{BL20_2}. 

\begin{figure}[h!]
\includegraphics[scale=.35]{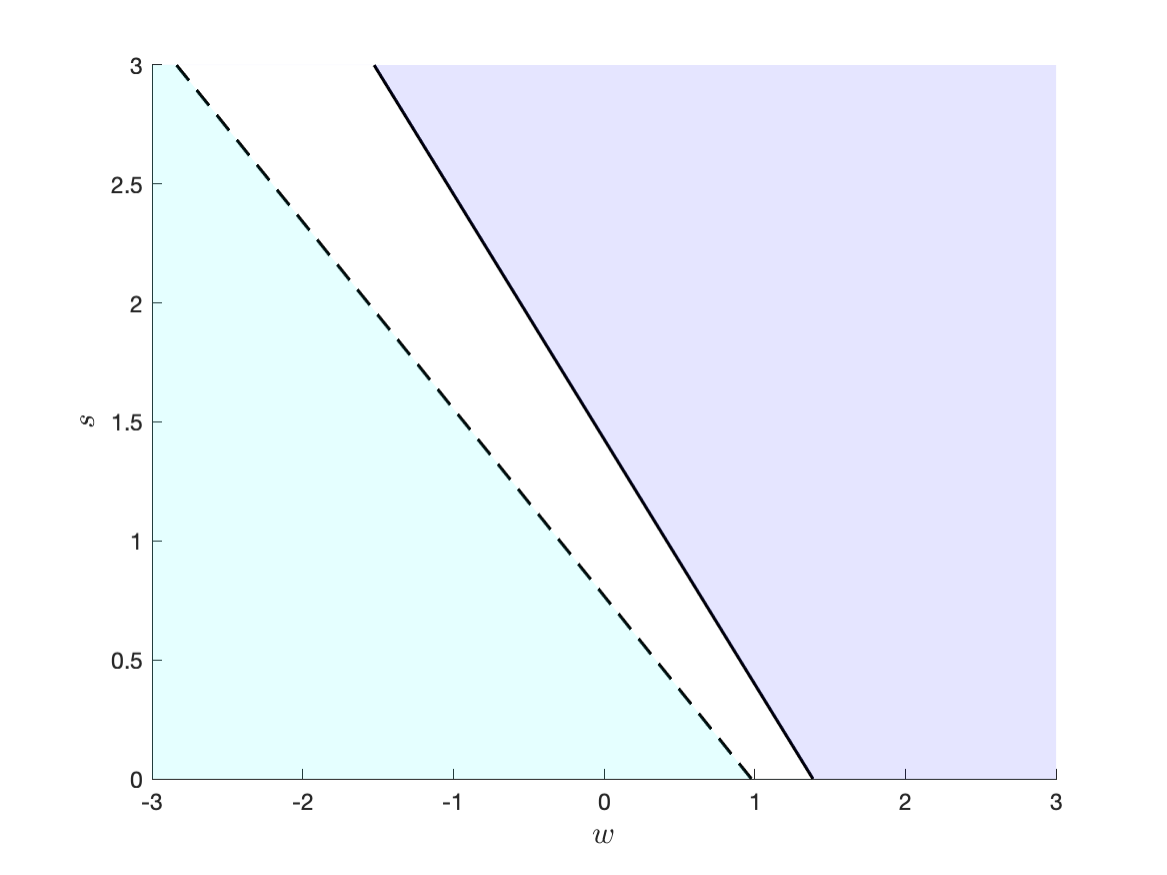}
\caption{Illustration of the \textcolor{violet}{sub-critical} $\{\calL_{s}^-(w,s) \geq 0 \}$ and \textcolor{cyan}{super-critical} $\{\calL_{s}^+(w,s) < 0 \}$ criteria for the attractive case}
\end{figure}

 %
 %
 %
 %
 %
 %
\section{Repulsive forcing and threshold with variable background}\label{sec_rep}

In this section, we aim to establish the critical thresholds for the repulsive EP system \eqref{main_sys}--\eqref{ini}. To this end, we conduct the phase plane analysis as introduced in Section \ref{ssec:rep}; for repulsive interaction, we recall the transformed ODE system \eqref{reduced_eq}:
\begin{equation}\label{rep_sys}
\lt\{\begin{split}
w' &= -\nu w + 1 - c(t)s \\
s' &= w,
\end{split}\rt.
\end{equation}
subject to initial point $(w_0,s_0)$. Here we assume that $0<c_- \le c(t) \le c_+$.

We first confirm the comparison principles. Note that
\bq\label{CT-fun}
\begin{split}
\frac{d}{dt}\calL_{P}(w(t),s(t)) &= -\nu w(t) + 1 -c(t)s(t) +\lt(\frac{\nu \sqrt{2P(s(t))} + 1 - c_1 s(t)}{\sqrt{2P(s(t))}}\rt)w(t) \\
&= \lt(\frac{1-c_1s(t)}{\sqrt{2P(s(t))}} \rt)\calL_{P}(w(t),s(t)) + (c_1-c(t))s(t).
\end{split}
\eq
Similarly, we observe that
\bq\label{CT-fun2}
\frac{d}{dt}\calL_{N}(w(t),s(t)) = \lt(\frac{c_2 s(t)-1}{\sqrt{2N(s(t))}} \rt)\calL_{N}(w(t),s(t)) + (c_2-c(t))s(t).
\eq
In particular, putting $(c_1,c_2)=(c_-,c_+)$ yields 
\[
\begin{split}
\calL_{P_-}(w(t),s(t)) \le \calL_{P_-}(w_0,s_0) \exp\lt(\int_0^t \frac{1-c_-s(\tau)}{\sqrt{2P_-(s(\tau))}} \,\dta \rt), \cr
\calL_{N_+}(w(t),s(t)) \ge \calL_{N_+}(w_0,s_0) \exp\lt(\int_0^t \frac{c_+s(\tau)-1}{\sqrt{2N_+(s(\tau))}} \,\dta \rt),
\end{split}
\]
provided that $s(\tau)$ satisfies $P_-(s(\tau))\ge 0$ and (resp. $N_+(s(\tau))\ge 0$) for all $\tau \in [0,t]$, and the
the weak comparison principle \eqref{CP:sup} follows. Similarly, by taking $(c_1,c_2)=(c_+,c_-)$, we obtain the strong comparison principle \eqref{CP:sub}.

The Lyapunov-based approach explained in Section \ref{ssec:rep} ---  essentially for the \emph{closing} condition \eqref{closing} --- makes use of the domains of $\sqrt{2P}$ and $\sqrt{2N}$. By witnessing the defining equation of $P$ and $N$ in \eqref{aux_P}, \eqref{aux_N}, we observe that the domains of $\sqrt{2P(\cdot)}$, $\sqrt{2N(\cdot)}$ correspond to the set of $s\in \R$ for which $P(s)\ge 0$, $N(s)\ge 0$, respectively. It implies that the endpoints of the domains are precisely determined by zeros of $P$ and $N$. The domains \eqref{dom} can be obtained by the explicit solutions for $\eqref{rep_sys}$.

We first investigate the exact formula of $\tilde{s}$ and $s_{**}$ in \eqref{str}. We take index $i \in \{1,2\}$. For $0\le\nu <2\sqrt{c_i}$ case, the solution $(w(t),s(t))$ of \eqref{rep_sys} with $c(t)= c_i>0$, $(w_0,s_0) = (0,a)$ is explicitly given as:
\begin{align}\label{ws:for}
\begin{aligned}
w(t) &= -\frac{1}{\mu_{i}}(\bar c a-1)\sin(\mu_{i} t )e^{-\frac{\nu t}{2}},\quad \mu_{i} = \sqrt{ c_i - \frac{\nu^2}{4}}, \cr
s(t) &= \frac{1}{\bar c}+\lt(a-\frac{1}{\bar c}\rt) \lt(\cos \mu_{i} t + \frac{\nu}{2\mu_{i}}\sin \mu_{i} t \rt)e^{-\frac{\nu t}{2}}.
\end{aligned}
\end{align}
To unveil the formula of $\tilde{s}$ \eqref{str}, we take $a=\tilde{s}$ and index $i=1$ so that 
\[
\calL_{P}(w_0,s_0) = 0 + \sqrt{2P(\tilde{s})} = 0,
\]
where $P$ satisfies \eqref{aux_P}. 
By \eqref{CT-fun}, we observe that
\[
\calL_{P}(w(t),s(t)) = w(t) + \sqrt{2P(s(t))} = 0, \quad t \in \lt[0,\frac{\pi}{\mu_{1}}\rt].
\]
Note that the range of $t$ can be obtained by the criteria $P(s(t))\ge0$ and \eqref{ws:for}.
In particular, we deduce that
\[
s\lt(\frac{\pi}{\mu_{i}}\rt) = 0 \ \ \Leftrightarrow \ \ 
\tilde{s} = s_0 = \frac{1+e^{\gamma_{1}}}{c_1}.
\]
Similarly, by taking $a=s_*$ and $i=2$, we obtain
\[
 \sqrt{2N(s(t))} = w(t) \ge 0 \quad \mbox{for} \quad t \in \lt[-\frac{\pi}{\mu_{2}},0\rt].
\]
By noting $s_0= s_*$, we infer that
\[
s_{**} = s\lt(-\frac{\pi}{\mu_{2}}\rt) =  \frac{1}{c_2} - \lt(s_*-\frac{1}{c_2}\rt)e^{\gamma_2}.
\]
This confirms \eqref{str}. 

\subsection{Finite-time breakdown}

In this subsection, we prove that finite breakdown occurs for the super-critical initial data. To this end, we present the following lemma, which is crucially utilized in obtaining an upper bound of the maximal time for which a local smooth solution persists.
\begin{lemma}\label{lem:time}
For the solution $(w(t),s(t))$ of \eqref{rep_sys}, the following identity holds:
\[
t_2-t_1 \le \int_{s(t_2)}^{s(t_1)}\frac{\ds}{\sqrt{2P_-(s)}},
\]
provided that 
\[
w_0 + \sqrt{2P_-(s_0)} \le 0
\]
and $s(\tau)\in\textnormal{Dom}(P_-)$ for all $\tau \in [t_1,t_2]$. Indeed, we have equality in the integral if $c(t)=c_-$ and $w_0 + \sqrt{2P_-(s_0)}=0$. 
\end{lemma}
\begin{proof}
By change of coordinates $s=s(t)$, we have
\[
\int_{s(t_2)}^{s(t_1)} \,\frac{\ds}{\sqrt{2P(s)}} =  \int_{t_2}^{t_1} \frac{w(t)}{\sqrt{2P(s(t))}}\, \dt \ge \int_{t_2}^{t_1} (-1)\, \dt =t_2-t_1,
\]
where the inequality comes from the weak comparison principle \eqref{CP:sup}. Under the latter assumption, we note that \eqref{CT-fun} turns the inequality into equality. This completes the proof.
\end{proof}
Several remarks are in order. Suppose that $P$ satisfies \eqref{aux_P} with $c_1=\bar c$.

\begin{remark}\label{rmk:int:ud}
For $0\le\nu <2\sqrt{\bar c}$ case, we obtain an explict integral formula:
\[
\int_{\textnormal{Dom}(P)} \,\frac{\ds}{\sqrt{2P(s)}} = \frac{\pi}{\mu_{\bar c}}, \ \ \textnormal{with} \ \ \mu_{\bar c}:=\sqrt{\bar c-\frac{1}{4}\nu^2}.
\]
The initial condition in \eqref{aux_P} can be switched to $P(a)=0$ for any $a\neq 1/{\bar c}$. In particular, for $\nu=0$ case, we obtain the so-called \textit{elliptic integral}:
\[
\int_{\frac{2}{\bar c}-a}^{a}\frac{\ds}{ \sqrt{(a-s)(\bar c(a+s)-2)}} = \frac{\pi}{\sqrt{\bar c}}.
\]
\end{remark}
\begin{remark}\label{rmk:int:od}
For $\nu \ge 2\sqrt{\bar c}$ case, we have
\[
\int_{0}^{s_0} \,\frac{\textnormal{d}s}{\sqrt{2P(s)}} \le \frac{\log{(\lambda s_0)} \vee 2}{\lambda}, \ \  \textnormal{with} \ \ \lambda:=\frac{\nu -\sqrt{\nu^2-4\bar c}}{2}>0,
\]
for any $s_0 > 0$. In particular, it implies that $\textnormal{Dom}(P)=[0,+\infty)$.
\end{remark}

\begin{proposition}\label{rep_blo_1}
Let $\nu\ge2\sqrt{c_{-}}$. For any $(w_0,s_0)\in\Sigma^{\sharp}=\{(w,s) : w\le -\sqrt{2P_-(s)}, \ s>0\}$, the solution to \eqref{rep_sys} with initial data $(w_0,s_0)$ satisfies $s(t_*)=0$ for some finite $t_*>0$.
\end{proposition}
\begin{proof} 
We claim that $s(t_*)=0$ with $t_*\le T_0$ where
\[
T_0:= \int_{0}^{s_0}\frac{\ds}{\sqrt{2P_-(s)}}.
\]
Suppose not, i.e., there exists $(w_0,s_0)\in\Sigma^{\sharp}$ such that $s(t)>0$ for all $0\le t \le t_*$ satisfying
$ t_* > T_0$.
Then, by weak comparison principle \eqref{CP:sup}, we have
\[
w(t) \le -\sqrt{2P_-(s(t))}, \ \ 0\le t \le t_*.
\]
Hence, by Lemma \ref{lem:time}, we note in passing
\[
t_* \le \int_{s(t_*)}^{s_0}\frac{\ds}{\sqrt{2P_-(s)}} \le T_0 < t_*,
\] 
which leads to a contradiction. Hence, the conclusion of the proposition follows by recalling that $T_0<+\infty$ via Remark \ref{rmk:int:od}.
\end{proof}

\begin{proposition}\label{rep_blo_2}
Suppose that $0\le \nu<2\sqrt{c_{-}}$. For any 
\[
(w_0,s_0)\in\Sigma^\sharp=\{(w,s): s>0\}\setminus\lt\{(w,s) :  -\sqrt{2P_-(s)}<w<\sqrt{2N_+(s)},s>0\rt\}, 
\]
the solution to \eqref{rep_sys} with initial data $(w_0,s_0)$ satisfies $s(t_*)=0$ for some finite $t_*>0$.
\end{proposition}
\begin{proof} 
We decompose $\Sigma^\sharp$ into three subsets:
\[
\Sigma^{\sharp}_{1}:=\lt\{w\le -\sqrt{2P_-(s)},\ 0<s\le s_* \rt\} \cup \{w\le0, \ s>s_*\},
\]
\[
\Sigma^\sharp_{2}:=\lt\{w>0, s \ge s_*\rt\}, \quad \mbox{and} \quad \Sigma^\sharp_{3}:=\lt\{ w \ge \sqrt{2N_+(s)}, 0<s<s_*\rt\}.
\]

$\bullet$ 
For $(w_0,s_0)\in \Sigma^{\sharp}_{1}$, we put $a:= \max\{s_0,s_*\}.$ Consider $P_*(\cdot)$ satisfying
\[
\frac{dP_*}{ds}= \nu\sqrt{2P_*(s)} + 1 - c _-s, \ P_*(a)=0.
\]

Then, we have
\[
\calL_{P_*}(w_0,s_0)=w_0+\sqrt{2P_*(s_0)} \le 0.
\]

Indeed, for $s_0 > s_*$, we have $w_0+\sqrt{2P_*(s_0)}= w_0 + \sqrt{2P_*(a)}=w_0 \le 0$, and for $0<s_0\le s_*$, we notice that $w_0+\sqrt{2P_*(s_0)} = w_0+\sqrt{2P_-(s_0)} \le 0$. The weak comparison principle \eqref{CP:sup} also holds for $P$ as well. Hence, following the argument in the proof of Proposition \ref{rep_blo_1} leads us to find that  $s(t_1)=0$ with 
\[
t_1 \le \int_0^a \frac{\ds}{\sqrt{2P(s)}}\le \frac{\pi}{\mu_-}, \ \ \mu_-:=\sqrt{c_- - \frac{\nu^2}{4}},
\]
where the last inequality follows from Remark \ref{rmk:int:ud}.

$\bullet$  For $(w_0,s_0)\in \Sigma^\sharp_2$, we prove that $w(t_2)=0$, $s(t_2)\ge s_*$ with $t_2 \le \frac{\pi/2}{\sqrt{c_-}}$. To this end, we temporarily introduce the following Lyapunov function, defined by
\[
\theta_{-}(w,s):= \arctan\lt(\frac{\sqrt{c_-}\lt(s-\frac{1}{c_-}\rt)}{w}\rt).
\] 
We claim that
\[
\frac{d}{dt}\theta_{-}((w(t),s(t))) \ge \sqrt{c_{-}} 
\]
as long as $s(t)\ge \frac{1}{c_{-}}$ and $w(t)>0$. It can be shown by the following straightforward computation:
\[
\begin{split}
\frac{d}{dt}\frac{\theta_{-}(w,s)}{\sqrt{c_-}}=&\frac{1}{1+c_-\lt(\frac{s-\frac{1}{c_-}}{w}\rt)^2}\lt(\frac{w^2-\lt(s-\frac{1}{c_-}\rt)\lt(-\nu w + 1 -c(t)s\rt)}{w^2}\rt) \\
 =&\frac{1}{1+c_-\lt(\frac{s-\frac{1}{c_-}}{w}\rt)^2}\lt(1+c_-\lt(\frac{s-\frac{1}{c_-}}{w}\rt)^2+\nu\lt(\frac{s-\frac{1}{c_-}}{w}\rt)+\frac{\lt(s-\frac{1}{c_-}\rt)(c(t)-c_-)s}{w^2}\rt)\ge 1.
\end{split}
\]
Hence, we find that $\theta_{-}(w(t_2),s(t_2))=\frac{\pi}{2}$ with 
\[
t_2\le \frac{\pi/2-\theta_{-}(w_0,s_0)}{\sqrt{c_{-}}} \le  \frac{\pi/2}{\sqrt{c_{-}}}.
\] 
In particular, we obtain $w(t_2) = 0$.

$\bullet$  For $(w_0,s_0) \in \Sigma^\sharp_3$, we argue as in the proof of Proposition \ref{rep_blo_1} that $s(t_3) = s_*$, $w(t_3)\ge 0$ with 
\[
t_3 \le \int_{s_0}^{s_*} \frac{\ds}{\sqrt{2N_+(s)}} \le \frac{\pi}{\mu_+}, \ \ \mu_+:=\sqrt{c_+ - \frac{\nu^2}{4}}.
\]
Henceforth, we deduce that, for any $(w_0,s_0)\in \Sigma^\sharp$,  $s(t_*)=0$ with 
\[
t_* \le \frac{\pi}{\mu_{-}}+\frac{\pi/2}{\sqrt{c_{-}}} + \frac{\pi}{\mu_{+}}.
\]
This delivers the finite upper bound of $t_*$  and completes the proof of the proposition.
\end{proof}

\begin{proof}[Proof of Theorem \ref{thm:rep:sup}]
The proof is based on a combination of the weak comparison principle \eqref{CP:sup} with Propositions \ref{rep_blo_1} and \ref{rep_blo_2}. For any regular initial data $(\rho_0,u_0)$ in the sense of Theorem \ref{thm_apriori}, if there exists $x\in\R$ such that
\[
\lt(\frac{u_0'(x)}{\rho_0(x)}, \frac{1}{\rho_0(x)}\rt) \in \Sigma^\sharp,
\]
then $\rho(\cdot,t)$ blows up in a finite time, and the conclusion of the theorem follows.
\end{proof}

\subsection{Global-in-time regularity} \label{ssec:rep:sub}

To establish the global regularity of the initial data in sub-critical regime, we prove the strong comparison principle \eqref{CP:sub} in the sequel. Recall from \eqref{CT-fun}, \eqref{CT-fun2} that if  $s(\tau)$ satisfies $P_+(s(\tau))\ge 0$ and (resp. $N_-(s(\tau))\ge 0$) for all $\tau \in [0,t]$, then we get
\bq\label{CP:sub:Gr}
\begin{split}
\calL_{P_+}(w(t),s(t)) \ge \calL_{P_+}(w_0,s_0) \exp\lt(\int_0^t \frac{1-c_+s(\tau)}{\sqrt{2P_+(s(\tau))}} \,\dta \rt), \cr
\calL_{N_-}(w(t),s(t)) \le \calL_{N_-}(w_0,s_0) \exp\lt(\int_0^t \frac{c_-s(\tau)-1}{\sqrt{2N_-(s(\tau))}} \,\dta \rt).
\end{split}
\eq
The proof for the strict inequalities in the strong comparison principle \eqref{CP:sub} requires us to control the integral in the exponential term in \eqref{CP:sub:Gr}. We demonstrate this in detail in the following propositions.

\begin{proposition}\label{rep_glo_1}
Let $\nu \geq 2\sqrt{c_{+}}$. Then, 
\[
\Sigma^\flat=\lt\{(w,s) : -\sqrt{2P_+(s)}<w, s>0\rt\}
\] 
is an invariant set for the system \eqref{rep_sys}.
\end{proposition}
\begin{proof}
We prove that for any $(w_0,s_0)\in\Sigma^\flat$, the solution to \eqref{rep_sys} satisfies $(w(t),s(t))\in\Sigma^\flat$ for all $t\ge 0$. Suppose  not, i.e., there exists $(w_0,s_0)\in\Sigma^\flat$ such that 
\[
T_*:=\inf\{t>0 : (w(t),s(t)) \notin  \Sigma^\flat\} < +\infty.
\]
In particular, we have $(w(T_*),s(T_*))\in \pa\Sigma^\flat$. We decompose the boundary of $\Sigma^\flat$ into two subsets:
\[
\pa \Sigma^\flat = \pa \Sigma^\flat_1 \cup \pa \Sigma^\flat_2:=\lt\{w+\sqrt{2P_+(s)}=0,\ s \ge 0 \rt\} \cup \lt\{w>0,\ s=0\rt\}.
\] 
We shall demonstrate that $T_*<+\infty$ leads to a contradiction. Indeed, the nature of proof is the same for $0<\nu<2\sqrt{c_+}$ case, which is slightly more complicated. Hence, we defer the details of the proof to the next proposition.
\end{proof}

\begin{proposition}\label{rep_glo_2}
Let $0<\nu<2\sqrt{c_{+}}$. Then, the following set
\[
\Sigma^\flat=\lt\{(w,s) : -\sqrt{2P_+(s)}<w<\sqrt{2N_-(s)} \rt\}
\]
is an invariant set for the system \eqref{rep_sys} provided that \eqref{ext_dam_11-22} holds.
\end{proposition}

\begin{proof}
We prove that for any $(w_0,s_0)\in\Sigma^\flat$, the solution to \eqref{rep_sys} satisfies $(w(t),s(t))\in\Sigma^\flat$ for all $t\ge 0$. Suppose, on the contrary, that there exists $(w_0,s_0)\in\Sigma^\flat$ such that 
\[
T_1:=\inf\{t>0 : (w(t),s(t)) \in \pa \Sigma^\flat\} < +\infty.
\]
We decompose the boundary of $\Sigma^\flat$ into three subsets in the sense of \eqref{bdy:sub} that
\[
\begin{split}
 &\pa \Sigma^\flat_1 \cup \pa\Sigma^\flat_2 \cup \pa\Sigma^\flat_3 \\
&\quad :=\lt\{\calL_{P_+}(w,s)=0, 0\le s < s_*\rt\} \cup \lt\{\calL_{N_-}(w,s)=0, 0 < s \le s_*\rt\} \cup \lt\{0 < w \le \sqrt{2N_-(0)}, s=0\rt\}.
\end{split}
\] 
We now consider the following sub-three cases.

\vspace{.2cm}
$\bullet$ {$(w(T_1),s(T_1))\in \pa \Sigma^\flat_1$.} We claim that there exists a small $\epsilon>0$ and $C_{\epsilon}>0$ such that
\bq\label{sub_cri:int}
\int_{T_1-\epsilon}^{T_1}\frac{1-c_{+}s(\tau)}{\sqrt{2P_+(s(\tau))}}\,\dta \ge -C_{\epsilon}.
\eq
If $s(T_1)<\frac{1}{c_{+}}$, then we can take small $\epsilon>0$ such that the integrand becomes non-negative for all $\tau \in \lt[T_1-\epsilon,T_1\rt]$. On the other hand, if $s(T_1)\ge \frac{1}{c_{+}}$, then we have $s(T_1) < s_*$ by the assumption. Since $\sqrt{2P_+(s_*)}=0$, we can take small $\epsilon>0$ such that $P_+(\tau)$ is away from zero for all $\tau \in \lt[T_1-\epsilon,T_1\rt]$. Hence, this supplies the estimate \eqref{sub_cri:int}. By combining this lower bound of integral  \eqref{sub_cri:int} with the inequality \eqref{CP:sub:Gr}, we find that
\[
w(T_1) + \sqrt{2P_+(s(T_1))} \ge \lt( w(T_1-\epsilon) + \sqrt{2P_+(s(T_1-\epsilon))}\rt)\exp(-C_{\epsilon})>0,
\]
leading to a contradiction.

\vspace{.2cm}

$\bullet$ {$(w(T_1),s(T_1))\in \pa \Sigma^\flat_2$.} Similarly, we can prove that there exist a small $\epsilon>0$ and $C_{\epsilon}>0$ such that
\[
\int_{T_1-\epsilon}^{T_1}\frac{c_{-}s(\tau)-1}{\sqrt{2N_-(s(\tau))}}\,\dta > -C_{\epsilon}.
\]
We find by means of the previous case with \eqref{CP:sub:Gr} that
\[
w(T_1) + \sqrt{2N_-(s(T_1))} \le \lt(w(T_1-\epsilon)+\sqrt{2N_-(s(T_1-\epsilon))}\rt)\exp(-C_{\epsilon})<0,
\]
which is a contradiction

\vspace{.2cm}

$\bullet$ {$(w(T_1),s(T_1))\in \pa \Sigma^\flat_3$.} Noting that $s(T_1)=0$ and $w(T_1)>0$ conveys us to the following
\[
0<w(T_1)=s'(T_1) = \displaystyle\lim_{t\to T_1-}\frac{s(T_1)-s(t)}{T_1-t}\le\limsup_{t\to T_1-}\frac{-s(t)}{T_1-t}\le0,
\]
that we arrive at the contradiction. This completes the proof.
\end{proof}

\begin{proof}[Proof of Theorem \ref{thm:rep:sub}] In view of the conclusion  that $\Sigma^{\flat}$ is \emph{invariant} for \eqref{rep_sys} from Proposition \ref{rep_glo_1} and \ref{rep_glo_2}, for any regular initial data $(\rho_0,u_0)$ satisfying
\[
\lt(\frac{u_0'(x)}{\rho_0(x)}, \frac{1}{\rho_0(x)}\rt) \in \Sigma^\flat, \quad \forall x\in\R,
\]
we find that the classical solution $(\rho,u)$ is indeed global by invoking \emph{a priori} estimate \eqref{est_X}. Hence, we arrive at the conclusion of the theorem.
\end{proof}

As a direct consequence of this result, we present the proof of Theorem \ref{thm:ion:glo}. 
\begin{proof}[Proof of Theorem \ref{thm:ion:glo}]
We begin by discussing the smallness assumption on $H_0$. In \cite{BCKpre},  the lower and upper bounds on the electric potential for a classical solution to the system \eqref{main_sys_ion}--\eqref{bdy} are investigated as follows:
\[
0 <  c_{-} \leq e^{\phi(t,x) }\leq c_{+}, \ \ \  c_{-} := \exp\lt(V^{-1}_-(H_0)\rt)\quad \mbox{and} \quad  c_{+} := \exp\lt(V^{-1}_+(H_0)\rt),
\]
where $V^{-1}_\pm$ are inverses of $V_\pm$ which is defined by
\begin{equation}\label{fctV}
V(z) := \left\{ \begin{array}{ll}
 \displaystyle V_+(z):= \int_0^z \sqrt{2U(r)}\,\textnormal{d}r & \textrm{for $z\geq0$},\\[3mm]
 \displaystyle V_-(z):= \int_z^0 \sqrt{2U(r)}\,\textnormal{d}r & \textrm{for $z\leq0$}.
  \end{array} \right.
\end{equation}
Note that $V_{\pm}(z)$ have the inverse functions and $V \in \calC^2(\R)$, see \cite[Section 2.1]{BCKpre}. 

Indeed, careful estimates via Taylor's expansion suggest that
\[
V_\pm(z) = \frac{1}{2}z^2 + \frac{1}{9}z^3 + \calO(z^4),
\]
so that
\[
V_\pm^{-1}(x) =  \pm \sqrt{2x} + {\mathcal O}(x), \qquad 0 < x\ll1.
\]
Hence, for small enough $H_0\ll1$, we have
\[
\frac{c_{+}}{c_{-}}= 1+ 2\sqrt{2H_0} + {\mathcal O}(H_0).
\]

We are now ready to investigate what secures the closing condition to hold; we go through as:
\[
\begin{split}
s_{**}\le 0 &\Longleftrightarrow \lt( (1+e^{\gamma_+})\frac{c_{-}}{c_+}-1\rt)e^{\gamma_-} \ge 1 \\
&\Longleftarrow 1 + e^{\gamma_+} \ge \frac{2c_+}{c_-} \\
&\Longleftarrow \gamma_+ \ge 2\lt(\frac{c_+}{c_-}-1\rt) = 4\sqrt{2H_0} + \calO(H_0) \\
&\Longleftarrow \nu\pi \ge \sqrt{4c_{+}-\nu^2}\lt(4\sqrt{2H_0} + \calO(H_0)\rt).
\end{split}
\]
Since $\sqrt{4c_{+}-\nu^2} = \calO(1)$ for small $\nu>0$, we find that \eqref{ci:sub} guarantees the closing condition as desired.
\end{proof}

%
%
%
%
%
%

\appendix

\section{A priori estimates in the proof of Theorem \ref{thm_apriori}} \label{app:lwp}
In this appendix, we provide the details of the a priori estimate \eqref{est:lwp}.

\medskip

$\bullet$ Estimates of $\|\rho -c\|_{H^s \cap \dot H^{-1}}$:  We find by means of the continuity equation \eqref{main_sys}${}_1$ and the dual formulation of $\dot{H}^{-1}$ that
\begin{align}\label{LWP1}
\frac{d}{dt}\|\rho - c\|_{\dot H^{-1}} &\leq \ \| \rho u\|_{L^2} + \|\pa_t c\|_{\dot H^{-1}}  \leq \|\rho\|_{L^\infty} \|u\|_{L^2} +\|\pa_t c\|_{\dot H^{-1}}.
\end{align}
For the higher-order estimates of $\rho-c$, we see that the continuity equation can be read as
\[
\pa_t (\rho - c) + \pa_x ((\rho - c)u) = - \pa_t c - \pa_x (cu).
\]
Differentiating it by $\pa_x^l$ for each $l=0,1,\dots, s$, we test $\pa_x^l(\rho-c)$ to find that
\begin{align*}
\frac12\frac{d}{dt}\intr |\pa_x^l (\rho - c)|^2\,\dx =\,& -\intr \pa_x^l (\rho - c) u \pa_x^{l+1} (\rho - c)\,\dx -\intr \pa_x^l (\rho - c) [\pa_x^{l+1}, u](\rho - c)\,\dx\cr
&- \intr \pa_x^l (\rho - c) \pa_x^l \pa_t c\,\dx - \intr \pa_x^l (\rho - c)   \pa_x^{l+1}(cu)\,\dx \cr
=:\,& I + II + III +  IV ,
\end{align*}
where $[\cdot, \cdot]$ stands for the commutator operator, i.e. $[A, B] = AB - BA$. Then, standard commutator estimates and Sobolev inequalities yield
\begin{align*}
I &\leq \|\pa_x u\|_{L^\infty}\| \pa_x^l (\rho - c)\|_{L^2}^2,\cr
II &\leq \| \pa_x^l (\rho - c)\|_{L^2} \|[\pa_x^{l+1}, u](\rho - c)\|_{L^2} \cr
&\leq  C\| \pa_x^l (\rho - c)\|_{L^2}\lt(\|\pa_x u\|_{L^\infty}\| \pa_x^l (\rho - c)\|_{L^2} + \|\pa_x^{l+1}u\|_{L^2}\|\rho - c\|_{L^\infty} \rt)\cr
III &\leq \|\pa_x^l \pa_t c\|_{L^2}\| \pa_x^l (\rho - c)\|_{L^2}, \quad \mbox{and}\cr
IV &\leq \| \pa_x^l (\rho - c)\|_{L^2} \|\pa_x^{l+1} (cu)\|_{L^2} \cr
&\leq C \| \pa_x^l (\rho - c)\|_{L^2}\lt( \|\pa_x^{l+1} c\|_{L^2}\|u\|_{L^\infty} + \|\pa_x^{l+1} u\|_{L^2}\|c\|_{L^\infty} \rt) \cr
&\leq C \| \pa_x^l (\rho - c)\|_{L^2}\lt( \|\pa_x^{l+1} c\|_{L^2}\|u\|_{H^1} + \|\pa_x^{l+1} u\|_{L^2}\|c\|_{L^\infty} \rt).
\end{align*}
Thus, we have
\begin{align}\label{LWP2}
\frac{d}{dt}\|\pa_x^l(\rho - c)\|_{L^2} \leq  &C\lt( \|\rho-c\|_{L^{\infty}}+\|c\|_{L^{\infty}} + \|\pa_x u\|_{L^\infty} \rt) \lt(\|\pa_x^l(\rho - c)\|_{L^2} + \|\pa_x^{l+1} u\|_{L^2} \rt) \cr
+ &C\|\pa_x^{l+1} c\|_{L^2}\|u\|_{H^1} +  \|\pa_x^l \pa_t c\|_{L^2}.
\end{align}

$\bullet$ Estimates of $\|u\|_{H^{s+1}}$:
We first estimate
\bq\label{LWP3}
\frac12\frac{d}{dt}\intr u^2\,\dx \leq \|\pa_x u\|_{L^\infty}\|u\|_{L^2}^2 + \|\pa_x \phi\|_{L^2}\|u\|_{L^2} = \|u\|_{L^2}\lt(\|\pa_x u\|_{L^\infty}\|u\|_{L^2} + \|\rho - c\|_{\dot{H}^{-1}}\rt).
\eq
For the higher-order estimates of $u$, where $l=0,1,\dots,s$, we deduce
\begin{align}\label{LWP4}
\frac12\frac{d}{dt}\intr |\pa_x^{l+1} u|^2\,\dx &\leq -\intr (\pa_x^{l+1} u) (\pa_x^{l+2} u) u\,\dx - \intr \pa_x^{l+1} u [\pa_x^{l+1}, u] \pa_x u\,\dx + k \intr \pa_x^{l+1} u \pa_x^l (\rho - c)\,\dx\cr
&\leq C\|\pa_x u\|_{L^\infty}\|\pa_x^{l+1} u\|_{L^2}^2 + \|\pa_x^{l+1} u\|_{L^2} \|\pa_x^l (\rho - c)\|_{L^2}.
\end{align}
Combining \eqref{LWP1}, \eqref{LWP2}, \eqref{LWP3}, and \eqref{LWP4}, we conclude the estimate \eqref{est:lwp}.


\end {document}